\documentclass[11pt, final, english]{article}

\usepackage[utf8]{inputenc}    
\usepackage[T1]{fontenc}       

\usepackage{amsmath}           
\usepackage{amssymb}          
\usepackage[amsmath,thmmarks,hyperref]{ntheorem} 
\usepackage{stmaryrd}

\usepackage{setspace}          
\usepackage{bbm}               
\usepackage{mathtools}         
\usepackage{mathrsfs}          

\usepackage[final=true,        
           pdfpagelabels       
           ]{hyperref}         
\usepackage[a4paper]{geometry} 
\usepackage[shortcuts]{extdash}
\usepackage{multicol}          

\allowdisplaybreaks

\newcommand{\refitem}[1] {\textit{\ref{#1}.)}}

\numberwithin{equation}{section}

\newcommand{\NN}{\mathbb{N}}
\newcommand{\RR}{\mathbb{R}}
\newcommand{\CC}{\mathbb{C}}
\newcommand{\Unit}{\mathbbm 1}

\newcommand{\I}{\mathrm{i}}

\newcommand{\cc}[1]{\overline{#1}}
\newcommand{\argument}{\,\cdot\,}
\newcommand{\hermitian}{\mathrm h}

\newcommand{\lie}[1]{\mathfrak{#1}}
\newcommand{\ring}[1]{\mathcal{#1}}
\newcommand{\variety}[1][]{\mathcal{X}_{#1}}
\newcommand{\A}{\ring{A}}
\newcommand{\B}{\ring{B}}
\newcommand{\subKnachKlammer}{{\hspace{-1pt}K}}
\DeclareMathOperator{\RE}{Re}
\DeclareMathOperator{\IM}{Im}

\newcommand{\Schwartz}{\mathscr{S}}
\newcommand{\Weyl}[1]{\mathcal W_{#1}}
\newcommand{\Vanish}[2][]{\mathcal V#1(#2#1)}
\newcommand{\ZerosNull}[2][]{\mathcal Z_0#1(#2#1)}
\newcommand{\Zeros}[2][]{\mathcal Z#1(#2#1)}
\newcommand{\Tensor}{\mathcal T}
\newcommand{\quotient}{\mathbin{/}}
\DeclareMathOperator{\pfaffian}{pf}
\DeclareMathOperator{\Pfaffian}{Pf}

\DeclareMathOperator{\ad}{ad}
\newcommand{\inv}{\mathrm{inv}}

\newcommand{\Dom}{\mathcal{D}}
\newcommand{\Adbar}{\mathcal{L}^*}
\newcommand{\UnivStar}{\mathcal{U}^*\hspace{-1pt}}

\DeclareMathOperator{\sign}{sign}
\newcommand{\Invariant}{E}
\newcommand{\Dual}{D}
\newcommand{\FinSet}{\mathcal{F}}
\newcommand{\FinSetMap}[2][]{\mathop{\FinSetMapMap}#1(#2#1)}
\newcommand{\FinSetMapMap}{\varepsilon}
\newcommand{\charStar}{\mathop{\mathrm{char}^*}}

\newcommand{\SymmetricGroup}{\mathfrak{S}}
\newcommand{\fthree}{\lie f_3}

\DeclarePairedDelimiter{\ordinarySet}{\{}{\}}
\DeclarePairedDelimiter{\ordinaryIP}{\langle}{\rangle}
\DeclarePairedDelimiter{\abs}{|}{|}
\DeclarePairedDelimiter{\norm}{\|}{\|}

\newcommand{\set}[3][]{\ordinarySet[#1]{\,#2 \;#1|\; #3\,}}
\newcommand{\skal}[3][]{\ordinaryIP[#1]{\,#2 \,#1|\, #3\,}}
\newcommand{\kom}[3][]{\mathop{#1[}#2\mathbin{,}#3\mathop{#1]}}
\newcommand{\genSAlg}[2][]{#1\langle\!#1\langle\,#2\,#1\rangle\!#1\rangle_{*\textup{-}\mathrm{alg}}}
\newcommand{\genSId}[2][]{#1\langle\!#1\langle\,#2\,#1\rangle\!#1\rangle_{*\textup{-}\mathrm{id}}}
\newcommand{\genVS}[2][]{#1\langle\!#1\langle\,#2\,#1\rangle\!#1\rangle_{\mathrm{lin}}}
\newcommand{\Kronecker}{\delta}

\newcommand{\mailto}[1]{\href{mailto:#1}{\texttt{#1}}}

\theoremheaderfont{\normalfont\bfseries}
\theorembodyfont{\itshape}
\newtheorem{lemma}{Lemma}[section]
\newtheorem{proposition}[lemma]{Proposition}
\newtheorem{theorem}[lemma]{Theorem}
\newtheorem{corollary}[lemma]{Corollary}
\newtheorem{definition}[lemma]{Definition}

\theoremheaderfont{\normalfont\bfseries}
\theorembodyfont{\normalfont}
\newtheorem{example}[lemma]{Example}

\theorembodyfont{\normalfont}
\makeatletter
\def\thmhead@plain#1#2#3{%
	\thmname{#1}\thmnumber{\@ifnotempty{#1}{ }\@upn{#2}}%
	\thmnote{ {\the\thm@notefont#3}}}
\let\thmhead\thmhead@plain
\makeatother

\theoremheaderfont{\scshape}
\theorembodyfont{\normalfont}
\theoremstyle{nonumberplain}
\theoremseparator{:}
\theoremsymbol{\hbox{$\boxempty$}}
\newtheorem{proof}{Proof}

\newtheorem{completionofproof}{Completion of Proof of Theorem~\ref{theorem}}

\theoremsymbol{}
\newtheorem{startofproof}{Start of Proof}

\author{
    \textbf{Philipp Schmitt}\\
    Institut für Analysis\\
    Leibniz Universität Hannover\\
    Am Welfengarten 1,
    30167 Hannover\\
    Germany\\
    \mailto{schmitt@math.uni-hannover.de}
    \and
    \textbf{Matthias Schötz}\footnote{Current address: 
      Fakulteta za matematiko in fiziko, Jadranska ulica 19, 1000 Ljubljana, Slovenia.
    }\\
    Instytut Matematyczny\\
    Polskiej Akademii Nauk\\
    ul.~\'Sniadeckich 8, 00-656 Warszawa\\
    Poland\\
    \mailto{schotz@impan.pl}
}

\title{Real Nullstellensatz for $2$-step nilpotent Lie algebras}

\date{October 2024}

\begin{document}
	\maketitle
	
	\begin{abstract}
		We prove a noncommutative real Nullstellensatz for $2$-step nilpotent Lie algebras that extends the classical,
		commutative real Nullstellensatz as follows: Instead of the real polynomial algebra $\RR[x_1,\dots,x_d]$ we consider the
		universal enveloping $^*$\=/algebra of a $2$-step nilpotent real Lie algebra (i.e.~the universal enveloping algebra of
		its complexification with the canonical $^*$\=/involution). Evaluation at points of $\RR^d$ is then generalized to
		evaluation through integrable $^*$\=/representations, which in this case are equivalent to filtered
		$^*$\=/algebra morphisms from the universal enveloping $^*$\=/algebra to a Weyl algebra. Our Nullstellensatz
		characterizes the common kernels of a set of such $^*$\=/algebra morphisms as the real ideals of the universal
		enveloping $^*$\=/algebra.
	\end{abstract}


\begin{onehalfspace}
	\tableofcontents

\section{Introduction}

The real Nullstellensatz \cite{risler:nullstellensatz} states that for any ideal $I$ of the real polynomial algebra
$\RR[x_1,\dots,x_d]$, $d\in \NN_0$, the following are equivalent:
\begin{enumerate}
  \item $I$ is real, i.e.~whenever $p_1, \dots, p_k \in \RR[x_1,\dots,x_d]$, $k\in \NN$, fulfil $\sum_{j=1}^k p_j^2 \in I$, then $p_1, \dots, p_k \in I$.
  \item $I$ is the vanishing ideal of its set of zeros $\mathcal{Z}_0(I) \coloneqq \set[\big]{\xi \in \RR^d}{p(\xi_1, \dots, \xi_d) = 0 \textup{ for all }p\in I}$,
    i.e.~$p\in \RR[x_1,\dots,x_d]$ fulfils $p(\xi_1, \dots,\xi_d) = 0$ for all $\xi \in \mathcal{Z}_0(I)$ if and only if $p\in I$.
\end{enumerate}
This Nullstellensatz, together with the Positivstellensatz of Krivine and Stengle \cite{krivine:AnneauxPreordonnes, stengle:aNullstellensatzAndAPositivstellensatzInSemialgebraicGeometry} (characterizing the convex cones of polynomials positive
on a set defined by polynomial inequalities), is one of the cornerstones of classical real algebraic geometry.

Since then it has become more and more clear that real algebraic geometry has close ties to topics in noncommutative algebra,
like operator theory and the representation theory of $^*$\=/algebras. For example, various Positivstellensätze for the polynomial
algebra have close analogs in the noncommutative world: roughly speaking, the real polynomial algebra as the object
of study is replaced by $^*$\=/algebras (i.e.\ complex algebras endowed with an antilinear involution $\argument^*$ that reverses the order of multiplication),
and evaluation at points of $\RR^d$ is replaced by (irreducible) $^*$\=/representations. It turns out that this way some Positivstellensätze
can be transferred to e.g.~matrices over commutative rings 
\cite{cimpric:realAlgebraicGeometryForMatricesOverCommutativeRings, gondard.ribenboim:le17eProblemeDeHilbertPourLesMatrices, 
hillar.nie:elementaryAndConstructiveSolutionToHilberts17thProblemForMatrices,
hol.scherer:matrixSumOfSquaresRelaxationsForRobustSemiDefinitePrograms, klep.schwaighfer:pureStatesPositiveMatrixPolynomialsAndSumsOfHermitianSquares},
freely generated $^*$\=/algebras \cite{helton:positiveNoncommutativePolynomialsAreSumsOfSquares, helton.klep.mcCullough.volcic:noncommutativePolynomialsDescribingConvexSets, helton.mcCullough:positivstellensatzForNonCommutativePolynomials},
the Weyl algebra \cite{schmuedgen:StrictPositivstellensatzForWeylAlgebra},
or the universal enveloping $^*$\=/algebra of real Lie algebras \cite{schmuedgen:StrictPositivstellensatzForEnvelopingAlgebras};
see also \cite{schmuedgen:nonCommutativeRealAlgebraicGeometry} for an overview.

With respect to noncommutative real Nullstellensätze for a $^*$\=/algebra $\A$ one has to distinguish between different versions:
One can develop ``one-sided'' Nullstellensätze characterizing the left ideal of elements $a\in \A$ that fulfil
$\pi(a) v = 0$ for certain $^*$\=/representations $\pi$ of $\A$ and vectors $v$ of the representation space of $\pi$,
and ``two-sided'' Nullstellensätze characterizing the ideal of elements $a\in \A$ such that $\pi(a) = 0$ for some $^*$\=/representations $\pi$.
At least heuristically, Nullstellensätze of the first type are stronger than those of the second type.
In any case, noncommutative Nullstellensätze are still rare, especially those that are not subject to additional restrictions:

For matrices over the real polynomial algebra there are one- and two-sided Nullstellensätze
\cite{cimpric:realAlgebraicGeometryForMatricesOverCommutativeRings, cimpric:realNullstellensatzForFreeModules},
and for quaternionic polynomials there is the real Nullstellensatz \cite{alon.paran:quaternionicNullstellensatz} and its
(one-sided) matrix version \cite{cimpric:matrixVersionsOfRealAndQuaternionicNullstellensaetze}.
Then there is the real Nullstellensatz for finite dimensional $^*$\=/representations \cite{procesi.schacher:nonCommutativeRealNullstellensatzAndHilberts17thProblem},
and a number of results for the free $^*$-algebra in variables $x_1, \dots,x_d,x_1^*, \dots,x_d^*$, $d\in \NN$,
which, however, are also subject to additional restrictions:
\cite{cimpric.helton.klep.mcCullough.nelson:onRealOneSidedIdealsInAFreeAlgebra, 
cimpric.helton.mcCullough.nelson:noncommutativeRealNullstellensatzCorrespondsToANoncommutativeRealIdealAlgorithms,
helton.mcCullough.Putinar:strongMajorizationInAFreeStarAlgebra}
provide a ``one-sided'' (real) Nullstellensatz for finitely-generated left ideals of the free ($^*$-)algebra,
and \cite{cimpric.helton.mccullogh.helton:realNullstellensatzAndStarIdeals} gives a ``two-sided'' real Nullstellensatz for $^*$\=/ideals
with finite codimension and for $^*$\=/ideals generated by homogeneous polynomials in only the unstared variables $x_1, \dots,x_d$. In all these cases it is sufficient
to consider only finite-dimensional $^*$\=/representations. Clearly, a real Nullstellensatz for the free $^*$-algebra
without restrictions would be of great value, but \cite[Sec.~3]{cimpric.helton.mccullogh.helton:realNullstellensatzAndStarIdeals}, \cite{popovych:onOStarRepresentabilityAndCStarRepresentabilityOfStarAlgebras}
also give an example demonstrating that there is no such general real Nullstellensatz,
even when taking into account arbitrary $^*$\=/representations on possibly infinite dimensional pre-Hilbert spaces. We therefore focus on
the smaller class of universal enveloping $^*$\=/algebras of real Lie algebras, which still have a very diverse representation theory.
In the following we present a first result in this direction, a ``two-sided'' real Nullstellensatz for $2$-step nilpotent Lie algebras,
Theorem~\ref{theorem}:

An ideal of the universal enveloping $^*$\=/algebra $\UnivStar(\lie g)$ of a $2$-step nilpotent real Lie algebra $\lie g$
is real (in an adequate sense for $^*$\=/algebras) if and only if it is the intersection of the kernel of Schrödinger 
representations of $\UnivStar(\lie g)$.

While $2$-step nilpotent real Lie algebras are rather trivial from the point of view of Lie theory, this Nullstellensatz
has several interesting features: First of all, it is easy to see that such a Nullstellensatz must take into account
infinite-dimensional, unbounded $^*$\=/representations, see Example~\ref{example:heisenberg} below. However, we do not 
allow arbitrary $^*$\=/representations, only sufficiently well-behaved ones, namely those that are ``integrable'' in the
sense of \cite[Sec.~9]{schmuedgen:invitationToStarAlgebras}, which in this setting are just Schrödinger representations.
Yet we have to consider Schrödinger representations with a varying number of position and momentum operators.
By Kirillov theory \cite{kirillov:UnitaryReps}, the representation space of a nilpotent real Lie algebra $\lie g$ is given by its coadjoint orbits, which
are the strata of a stratification of the dual vector space $\lie g^*$, see e.g.\ \cite{corwin.greenleaf:representationTheoryOfNilpotentLieGroupsAndTheirApplications}.
This non-trivial geometry is already present in
the $2$-step nilpotent case. The proof of Theorem~\ref{theorem} thus is a preparation for the general
real Nullstellensatz for arbitrary nilpotent Lie algebras that we hope to present in a follow-up. In contrast to the
general case, the $2$-step nilpotent case can be solved entirely with algebraic tools and without deeper knowledge of Kirillov theory.

In the next Section~\ref{sec:preliminaries} we fix our notation, and then in Section~\ref{sec:theorem} we give the precise 
statement of Theorem~\ref{theorem} and develop its proof. In the final Section~\ref{sec:examples} we discuss $2$ instructive
examples: the Heisenberg Lie algebra and the $2$-step nilpotent Lie algebra $\fthree$ that is freely generated by $3$ elements.
The example of the Heisenberg Lie algebra summarizes the main ideas of the proof of Theorem~\ref{theorem},
the example of $\fthree$ demonstrates some of the pathologies that arise in the general case.

\section{Notation and preliminaries}\label{sec:preliminaries}

\subsection{\texorpdfstring{$^*$-Algebras}{*-Algebras} and their ideals}

A \emph{$^*$\=/algebra} $\A$ is a complex associative algebra with unit $\Unit$ and endowed with an antilinear involution $\argument^* \colon \A \to \A$
such that $(ab)^* = b^* a^*$ for all $a,b\in \A$. For example, if $\Dom$ is a pre-Hilbert space, then $\Adbar(\Dom)$,
the algebra of adjointable endomorphisms of $\Dom$ with $\argument^* \colon \Adbar(\Dom) \to \Adbar(\Dom)$ assigning
to every $a\in \Adbar(\Dom)$ its adjoint endomorphism $a^* \in \Adbar(\Dom)$, is a $^*$\=/algebra.
An element $a\in \A$ that fulfils $a^* = a$ is called \emph{hermitian}, and $a\in \A$ is called \emph{antihermitian} 
if $a^* = -a$. On every $^*$\=/algebra $\A$ we define the \emph{commutator bracket} $\kom{\argument}{\argument} \colon \A \times \A \to \A$,
\begin{equation}
  (a,b)\mapsto {\kom{a}{b}} \coloneqq ab-ba
  .
\end{equation}
Note that the antihermitian elements of $\A$ are closed under this commutator bracket, in contrast to the hermitian elements.
A \emph{$^*$\=/algebra morphism} $\Phi \colon \A \to \B$ between two $^*$\=/algebras
$\A$ and $\B$ is a unital algebra morphism such that $\Phi(a^*) = \Phi(a)^*$ for all $a\in \A$.
In this case $\A$ is called the \emph{domain} of $\Phi$ and $\B$ its \emph{codomain}.
In particular, a \emph{$^*$\=/character} on a $^*$\=/algebra $\A$ is a $^*$\=/algebra morphism $\varphi \colon \A \to \CC$
and we write $\charStar(\A)$ for the set of all $^*$\=/characters on $\A$. 

A $^*$\=/subalgebra of a $^*$\=/algebra $\A$ is a subset $\B$ of $\A$ with $\Unit \in \B$ that is closed under addition and multiplication
and stable under the $^*$\=/involution (thus $\B$ is again a $^*$\=/algebra). For any subset $S$ of $\A$, the $^*$\=/subalgebra
\emph{generated by $S$} is the smallest (with respect to inclusion $\subseteq$) $^*$\=/subalgebra of $\A$ that contains $S$,
and it is denoted by $\genSAlg{S}$.

For $^*$\=/algebras, the notion of a real ideal has to be adapted. While one can find different notions of real one-sided
ideals in the literature, for the two-sided case they all coincide.
\begin{definition} \label{definition:realIdeal}
  Let $I$ be an ideal of a $^*$\=/algebra $\A$. Then we say that $I$ is \emph{real} if the following holds:
  whenever $a_1, \dots, a_k \in \A$, $k\in \NN$, fulfil $\sum_{j=1}^k a_j^* a_j \in I$, then $a_1, \dots, a_k \in I$.
\end{definition}
For noncommutative $^*$\=/algebras it is a non-trivial task to give an explicit description of the smallest real ideal
that contains a given ideal, see \cite{cimpric.helton.mcCullough.nelson:noncommutativeRealNullstellensatzCorrespondsToANoncommutativeRealIdealAlgorithms}
for the one-sided case.

An ideal $I$ of a $^*$\=/algebra $\A$ that is stable under the $^*$\=/involution, i.e.~$a^* \in I$ for all $a\in I$, is called a \emph{$^*$\=/ideal}.
Every real ideal $I$ of a $^*$\=/algebra $\A$ is a $^*$\=/ideal: indeed, given $a\in I$, then $(a^*)^* (a^*) = a a^* \in I$ and therefore $a^* \in I$. 
Moreover, the kernel $\ker \Phi = \set{a\in \A}{\Phi(a) = 0}$ of every $^*$\=/algebra morphism $\Phi \colon \A \to \B$
into any $^*$\=/algebra $\B$ is a $^*$\=/ideal of $\A$. Yet if $\B = \Adbar(\Dom)$
for some pre-Hilbert space $\Dom$, then
$\ker \Phi$ is even a real ideal: indeed, given $a_1,\dots,a_k \in \A$, $k\in \NN$ such that $\Phi\bigl( \sum_{j=1}^k a_j^* a_j \bigr) = 0$,
then $\sum_{j=1}^k \skal{\Phi(a_j)(\psi)}{\Phi(a_j)(\psi)} = 0$ for all $\psi \in \Dom$ and therefore $\norm{\Phi(a_j)(\psi)} = 0$
for all $\psi \in \Dom$ and $j \in \{1,\dots,k\}$, i.e.~$\Phi(a_1) = \dots = \Phi(a_k) = 0$.

In order to formulate a real Nullstellensatz for $^*$\=/algebras we also have to adapt the notion of a vanishing ideal.
The usual approach is to replace evaluation of polynomials $p\in \RR[x_1,\dots,x_d]$ at points $\xi \in \RR^d$ by
application of a (suitably well-behaved, possibly irreducible) $^*$\=/representation. We will essentially follow
this approach, but for the time being we define:
\begin{definition} \label{definition:vanishingIdeal}
  Let $\A$ be a $^*$\=/algebra and $S$ a set of $^*$\=/algebra morphisms with domain $\A$. Then the \emph{vanishing ideal} of $S$ is
  $\Vanish{S} \coloneqq \bigcap_{\Phi \in S} \ker \Phi$. For the empty set this is understood as $\Vanish{\emptyset} \coloneqq \A$.
\end{definition}
We can use $^*$\=/characters in order to reformulate the classical, commutative real Nullstellensatz:

\begin{lemma} \label{lemma:commutativeNullstellensatz}
  Let $\A$ be a $^*$\=/algebra that is generated by finitely many elements and let $I$ be a real ideal of $\A$. 
  Write $\ZerosNull{I} \coloneqq \set[\big]{\varphi \in \charStar(\A)}{ I \subseteq \ker \varphi }$.
  If $\kom{a}{b} \in I$ for all $a,b\in\A$, then $I= \Vanish{\ZerosNull{I}}$.
\end{lemma}
\begin{proof}
  It is immediately clear that $I \subseteq \Vanish{\ZerosNull{I}}$, but for the converse inclusion we have to apply the commutative real Nullstellensatz:

  By passing to real and imaginary parts $\RE(a) \coloneqq (a + a^*)/2$ and $\IM(a) \coloneqq -\I (a - a^*)/2$ we can assume that $\A$ is generated by
  finitely many hermitian elements $g_1, \dots, g_{d}$ of $\A$ that fulfil
  $\kom{g_j}{b} \in I$ for all $j\in \{1,\dots,d\}$ and all $b\in \A$. As $I$ is a real ideal of $\A$ it is a $^*$\=/ideal
  so that we can construct the quotient $^*$\=/algebra $\A \quotient I$, which is generated by central hermitian elements
  $[g_1], \dots, [g_{d}]$, where $[\argument] \colon \A \to \A \quotient I$ is the canonical projection onto the quotient.
  Therefore $\A \quotient I$ is commutative, so its hermitian elements form a finitely generated commutative real algebra
  that we denote by $(\A \quotient I)_\hermitian$, and the canonical morphism of $\RR$-algebras
  $\check \argument \colon \RR[x_1,\dots,x_{d}] \to (\A \quotient I)_\hermitian$,
  \begin{equation*}
    p \mapsto \check p \coloneqq p\bigl([g_1], \dots, [g_{d}]\bigr)
  \end{equation*}
  is well-defined and surjective.
  
  We check that its kernel $\ker \check \argument$ is a real ideal of $\RR[x_1,\dots,x_{d}]$ as in the commutative real Nullstellensatz
  \cite{risler:nullstellensatz} cited in the introduction: If $p_1, \dots, p_k \in \RR[x_1,\dots,x_{d}]$ fulfil
  $\sum_{j=1}^k p_j^2 \in \ker \check \argument$, then $\sum_{j=1}^k \check p_j^* \check p_j = 0$ and so there
  exist representatives $a_1, \dots, a_k \in \A$ fulfilling $[a_j] = \check p_j$ for all $j\in \{1,\dots,k\}$.
  Then $\sum_{j=1}^k [a_j]^* [a_j] = 0$, i.e.~$\sum_{j=1}^k a_j^* a_j \in I$. As the ideal $I$ of $\A$ is real
  in the sense of Definition~\ref{definition:realIdeal}, this implies $a_1, \dots, a_k \in I$, so $\check p_j = [a_j] = 0$ for all $j\in \{1,\dots,k\}$,
  i.e.~$p_1, \dots, p_k \in \ker \check\argument$. By the commutative real Nullstellensatz \cite{risler:nullstellensatz},
  $\ker \check\argument$ is the vanishing ideal of its set of zeros 
  $\mathcal{Z}'_0(\ker \check\argument) \coloneqq \set[\big]{\xi \in \RR^{d}}{p(\xi_1, \dots, \xi_{d}) = 0 \textup{ for all }p\in \ker\check\argument}$.
  For every $\xi \in \mathcal{Z}'_0(\ker \check\argument)$ the corresponding evaluation functional $\delta_\xi \colon \RR[x_1,\dots,x_{d}] \to \RR$,
  $p \mapsto \delta_\xi(p) \coloneqq p(\xi_1,\dots,\xi_{d})$ descends to a well-defined morphism of $\RR$-algebras
  $\check \delta_\xi \colon (\A \quotient I)_\hermitian \to \RR$, $\check p \mapsto \check\delta_\xi(\check p) \coloneqq \delta_\xi(p)$,
  which can be extended to a $^*$\=/algebra morphism from $\A \quotient I$ to $\CC$, i.e.~a $^*$\=/character on $\A \quotient I$, which in turn
  can be pulled back to a $^*$\=/character $\varphi_\xi \colon \A \to \CC$. By construction $\varphi_\xi \in \ZerosNull{I}$ and 
  $\varphi_\xi(a) = p(\xi)$ holds for all $a\in \A$ with $a=a^*$ and all $p \in \RR[x_1,\dots,x_{d}]$ with $\check p = [a]$.
  
  Now consider a hermitian element $a$ of $\Vanish{\ZerosNull{I}}$, so $\varphi(a) = 0$ for all $\varphi \in \ZerosNull{I}$. Then there exists
  $p\in \RR[x_1,\dots,x_{d}]$ such that $\check p = [a]$, and $p(\xi) = \varphi_\xi(a) = 0$ for all $\xi \in \mathcal{Z}'_0(\ker \check\argument)$.
  Therefore $\check p = 0$ by the commutative real Nullstellensatz and consequently $a\in I$. A general element $a\in \Vanish{\ZerosNull{I}}$
  can be decomposed as a linear combination of its hermitian real and imaginary part, $a = \RE(a) + \I \IM(a)$,
  that fulfil $\varphi(\RE(a)) = \varphi(\IM(a)) = 0$ for all $\varphi \in \ZerosNull{I}$, so again $a \in I$.
\end{proof}

\subsection{Universal enveloping \texorpdfstring{$^*$-algebras}{*-algebras} of real Lie algebras}

Consider a complex vector space $V$ and its tensor algebra $\Tensor(V) \coloneqq \bigoplus_{k=0}^\infty V^{\otimes k}$,
then any antilinear involution $\argument^* \colon V \to V$ can be extended in a unique way to an antilinear involution
$\argument^* \colon \Tensor(V) \to \Tensor(V)$ with which $\Tensor(V)$ becomes a $^*$\=/algebra.

Let $\lie g$ be a real Lie algebra. Its complexification $\lie g \otimes \CC$
is a complex Lie algebra with Lie bracket ${\kom{\argument}{\argument}} \colon (\lie g \otimes \CC) \times (\lie g \otimes \CC) \to \lie g \otimes \CC$,
$\bigl( X \otimes \lambda, Y \otimes \mu \bigr) \mapsto {\kom{X\otimes \lambda}{Y\otimes \mu}} \coloneqq {\kom{X}{Y}} \otimes \lambda\mu$.
We define the antilinear involution $\argument^* \colon \lie g \otimes \CC \to \lie g \otimes \CC$,
\begin{equation}
  X \otimes \lambda \mapsto (X\otimes \lambda)^* \coloneqq - X \otimes \cc{\lambda}
  ,
\end{equation}
where $\cc\argument$ denotes complex conjugation; thus $\Tensor(\lie g \otimes \CC)$ is a $^*$\=/algebra.
Note that ${\kom{V}{W}}^* = {\kom{W^*}{V^*}}$ for all $V,W \in \lie g \otimes \CC$.
Let $J_{\lie g \otimes \CC}$ be the ideal of $\Tensor(\lie g \otimes \CC)$ generated by
\begin{equation}
  \set[\big]{V \otimes W - W \otimes V - {\kom{V}{W}} }{V,W\in \lie g \otimes \CC}
  .
\end{equation}
One can check that $J_{\lie g \otimes \CC}$ is even a $^*$\=/ideal so that the quotient $\Tensor(\lie g \otimes \CC) / J_{\lie g\otimes \CC}$
becomes a $^*$\=/algebra:

\begin{definition} \label{definition:iota}
  Let $\lie g$ be a real Lie algebra, then its \emph{universal enveloping $^*$\=/algebra} $\UnivStar(\lie g)$ is the
  quotient $^*$\=/algebra $\Tensor(\lie g \otimes \CC) \quotient J_{\lie g}$. Moreover we define the map $\iota \colon \lie g \to \UnivStar(\lie g)$,
  \begin{equation}
    X \mapsto \iota(X) \coloneqq [X\otimes 1]
  \end{equation}
  where $[\argument] \colon \Tensor(\lie g \otimes \CC) \to \UnivStar(\lie g)$ is the canonical projection onto the quotient.
\end{definition}
Clearly $\iota \colon \lie g \to \UnivStar(\lie g)$ is $\RR$-linear and maps $\lie g$ to the antihermitian elements
of $\UnivStar(\lie g)$ of degree $1$, and ${\kom[\big]{\iota(X)}{\iota(Y)}} = \iota\bigl(\kom{X}{Y}\bigr)$ holds for all $X,Y\in \lie g$.
By construction, the following universal property is fulfilled:

\begin{proposition} \label{proposition:constructionOfRep}
  Assume $\lie g$ is a real Lie algebra and $\A$ a $^*$\=/algebra. Let $\phi \colon \lie g \to \A$
  be an $\RR$-linear map that fulfils $\phi(X)^* = -\phi(X)$ for all $X\in \lie g$ and 
  ${\kom[\big]{\phi(X)}{\phi(Y)}} = \phi\bigl(\kom{X}{Y}\bigr)$ for all $X,Y\in \lie g$. Then there exists a unique
  $^*$\=/algebra morphism $\Phi \colon \UnivStar(\lie g) \to \A$ such that $\Phi \circ \iota = \phi$.
\end{proposition}
In many explicit formulas we will, however, drop this (injective) map $\iota$ and by abuse of notation treat elements of $\lie g$ directly
as elements of $\UnivStar(\lie g)$.

\subsection{The Weyl algebra}

For $d\in \NN_0$ the $d$-dimensional Weyl algebra is constructed in a similar way: First we equip the complex vector space $\CC^{2d}$
with the antilinear involution $\argument^* \colon \CC^{2d} \to \CC^{2d}$ of componentwise complex conjugation;
thus $\Tensor(\CC^{2d})$ is a $^*$\=/algebra. We write $p_1, \dots, p_d, q_1,\dots, q_d$ for the standard basis of $\CC^{2d}$
and write $\omega \colon \CC^{2d} \times \CC^{2d} \to \CC$ for the standard symplectic form, i.e.~$\omega(p_k,p_\ell) = \omega(q_k,q_\ell) = 0$
and $\omega(p_k, q_\ell) = - \Kronecker_{k,\ell}$ where $\Kronecker_{k,\ell}$ is the Kronecker-$\delta$, i.e.~$\Kronecker_{k,\ell} \coloneqq 1$
if $k=\ell$ and otherwise $\Kronecker_{k,\ell} \coloneqq 0$.
Let $J_d$ be the ideal of $\Tensor(\CC^{2d})$ that is generated
by the set $\set[\big]{ v \otimes w - w \otimes v - \I \omega(v,w) \Unit}{v,w\in \CC^{2d}}$. Then $J_d$ is a $^*$\=/ideal
and therefore the quotient algebra $\Tensor(\CC^{2d}) \quotient J_d$ becomes a $^*$\=/algebra:

\begin{definition}
  For all $d\in \NN_0$ the \emph{Weyl algebra} $\Weyl{d}$ is the quotient $^*$\=/algebra $\Tensor(\CC^{2d}) \quotient J_d$.
  We also define 
  \begin{equation}
    P_j \coloneqq [p_j]\quad\quad\text{and}\quad\quad Q_j \coloneqq [q_j]
  \end{equation}
  for all $j\in \{1,\dots,d\}$, where $[\argument] \colon \Tensor(\CC^{2d}) \to \Weyl{d}$ is the canonical projection onto the quotient.
\end{definition}
Note that $P_1, \dots,P_d$ and $Q_1, \dots, Q_d$ are hermitian elements of $\Weyl{d}$. Clearly $\Weyl{0} \cong \CC$ and
\begin{equation}
  {\kom{P_k}{P_\ell}} = {\kom{Q_k}{Q_\ell}} = 0
  \quad\quad\text{and}\quad\quad
  {\kom{P_k}{Q_\ell}} = -\I \Kronecker_{k,\ell} \Unit
  \quad\quad\text{for all $k,\ell \in \{1,\dots,d\}$.}
\end{equation}

Recall that a $2$-step nilpotent Lie algebra $\lie g$ is a Lie algebra that fulfils ${\kom[\big]{{\kom{X}{Y}}}{Z}} = 0$ for 
all $X,Y,Z \in \lie g$. For $2$-step nilpotent Lie algebras the Jacobi identity is trivially fulfilled, i.e.~every
antisymmetric bilinear map ${\kom{\argument}{\argument}} \colon \lie g \times \lie g \to \lie g$ that fulfils
${\kom[\big]{{\kom{X}{Y}}}{Z}} = 0$ for all $X,Y,Z \in \lie g$ provides a $2$-step nilpotent Lie algebra.

\begin{definition} \label{definition:varietyBullet}
  Let $\lie g$ be a $2$-step nilpotent real Lie algebra and $d\in \NN_0$. 
  Let $\Weyl{d}^{(1)}$ be the linear subspace of the Weyl algebra $\Weyl{d}$ that is generated by the unit $\Unit$ of $\Weyl{d}$ together
  with $P_1,\dots,P_d, Q_1, \dots, Q_d \in \Weyl{d}$. A $^*$\=/algebra morphism
  $\Phi \colon \UnivStar(\lie g) \to \Weyl{d}$ is called \emph{filtered} if $\Phi(X) \in \Weyl{d}^{(1)}$ for all $X \in \lie g$
  and we write $\variety[d]$ for the set of all filtered $^*$\=/algebra morphisms $\Phi \colon \UnivStar(\lie g) \to \Weyl{d}$.
  The union of these sets will be denoted by $\variety[\bullet] \coloneqq \bigcup_{d=0}^\infty \variety[d]$.
\end{definition}
In the real Nullstellensatz for $2$-step nilpotent Lie algebras this set $\variety[\bullet]$ will take the place
of the affine space $\RR^d$ (or rather: of evaluations at points of $\RR^d$) in the classical real Nullstellensatz for
the polynomial algebra. Following \cite{schmuedgen:StrictPositivstellensatzForEnvelopingAlgebras} we should actually
use (irreducible) integrable $^*$\=/representations of $\UnivStar(\lie g)$. For a finite-dimensional $2$-step nilpotent Lie algebra
$\lie g$, however, every element $\Phi \in \variety[d]$, $d\in \NN_0$ yields such an integrable $^*$\=/representation by
composition with the Schrödinger representation $\pi_d \colon \Weyl{d} \to \Adbar\bigl( \Schwartz(\RR^d)\bigr)$ on the
Schwartz space of rapidly decreasing functions $\Schwartz(\RR^d)$, i.e.
\begin{equation}
  \pi_d(P_j)(\psi)(\xi) = \frac{1}{\I}\frac{\partial \psi}{\partial x_j} (\xi)
  \quad\quad\text{and}\quad\quad
  \pi_d(Q_j)(\psi)(\xi) = \xi_j \psi(\xi)
\end{equation}
for all $\psi \in \Schwartz(\RR^d)$, $\xi \in \RR^d$, and $j\in \{1,\dots,d\}$. For $d=0$ this is the trivial $^*$\=/representation
$\pi_0 \colon \Weyl{0} \to \Adbar(\CC) \cong \CC$. Conversely, every irreducible integrable $^*$\=/representations of 
$\UnivStar(\lie g)$ factors in this way through a Weyl algebra $\Weyl{d}$, but we will not use this fact.
We only make the following observation after which there is no more need to consider $^*$\=/representations any further:

\begin{proposition} \label{proposition:vanishingIsReal}
  Let $\lie g$ be a $2$-step nilpotent real Lie algebra and $S \subseteq \variety[\bullet]$. Then the vanishing ideal
  $\Vanish{S}$ is a real ideal of $\UnivStar(\lie g)$.
\end{proposition}
\begin{proof}
  It is easy to check that an arbitrary intersection of real ideals of $\UnivStar(\lie g)$ is again a real ideal,
  so we only have to check that for every $\Phi \in S$ the ideal $\ker \Phi$ of $\UnivStar(\lie g)$ is real.
  Let $\Phi \in S$ be given, then there is $d\in \NN_0$ such that $\Phi \in \variety[d]$, i.e.~$\Phi$ is a filtered
  $^*$\=/morphism $\UnivStar(\lie g) \to \Weyl{d}$. It is well-known that the Schrödinger representation
  $\pi_d \colon \Weyl{d} \to \Adbar\bigl( \Schwartz(\RR^d)\bigr)$ is injective (in fact, $\Weyl{d}$ is a simple $^*$\=/algebra),
  so $\ker \Phi = \ker (\pi_d \circ \Phi)$, which is a real ideal of $\UnivStar(\lie g)$ because $\pi_d \circ \Phi$ maps
  to a $^*$\=/algebra of adjointable endomorphisms on a pre-Hilbert space.
\end{proof}

\begin{example} \label{example:heisenberg}
  The Heisenberg Lie algebra $\lie h$ is the $2$-step nilpotent real Lie algebra with basis $X,Y,Z \in \lie h$ and with the
  Lie bracket that fulfils $\kom{X}{Y} = Z$ and $\kom{X}{Z} = \kom{Y}{Z} = 0$. For $\lambda \in \RR\setminus \{0\}$ the linear map
  $\phi_\lambda \colon \lie h \to \Weyl{1}$ that is given by $\phi_\lambda(X) \coloneqq \I P_1$, $\phi_\lambda(Y) \coloneqq \I \lambda Q_1$, 
  $\phi_\lambda(Z) \coloneqq \I \lambda$ extends to a $^*$\=/algebra morphism $\Phi_\lambda \colon \UnivStar(\lie h) \to \Weyl{1}$
  by Proposition~\ref{proposition:constructionOfRep}, and $\Phi_\lambda$ clearly is filtered, so $\Phi_\lambda \in \variety[1]$.
  Similarly, for all $\xi,\eta \in \RR$ the linear map $\psi_{\xi,\eta} \colon \lie h \to \Weyl{0}$ that is given by
  $\psi_{\xi,\eta}(X) \coloneqq \I \xi$, $\psi_{\xi,\eta}(Y) \coloneqq \I \eta$, and $\psi_{\xi,\eta}(Z) \coloneqq 0$
  also extends to a $^*$\=/algebra morphism $\Psi_{\xi,\eta} \colon \UnivStar(\lie h) \to \Weyl{0}$
  by Proposition~\ref{proposition:constructionOfRep}, and $\Psi_{\xi,\eta}$ is trivially filtered because $\Weyl{0}^{(1)} = \Weyl{0}$.
  So $\Psi_{\xi,\eta} \in \variety[0]$.
  Note that $Z$ is in the kernel of all the filtered $^*$\=/algebra morphisms of this second type, which take values in
  $\Weyl{0} \cong \CC$. This is an instance of a much more general principle:
  Define the inner derivation $\ad_X \colon \UnivStar(\lie h) \to \UnivStar(\lie h)$,
  $a \mapsto \ad_X(a) \coloneqq {\kom{X}{a}}$, then $(\ad_X)^k(Y^k) = k! Z^k$ for all $k\in \NN$.
  Assume $\Phi \colon \UnivStar(\lie h) \to \A$ is a $^*$\=/algebra morphism into a $C^*$\=/algebra $\A$.
  Then
  \begin{equation}
    \norm[\big]{\Phi(Z^k)}^{\frac{1}{k}}
    =
    \biggl(\frac{1}{k!}\norm[\big]{\Phi\bigl((\ad_X)^k(Y^k)\bigr)}\biggr)^{\frac{1}{k}}
    \le
    \frac{2 \,\norm{\Phi(X)}\, \norm{\Phi(Y)} }{(k!)^{1/k}} 
    \xrightarrow{k\to\infty}
    0
  \end{equation}
  holds. Moreover, $\norm{\Phi(Z)^k} = \norm{\Phi(Z)}^k$ because $Z$ is antihermitian, hence normal,
  so this estimate shows that $\norm{\Phi(Z)} = 0$, i.e.~$\Phi(Z) = 0$
  for every $^*$\=/algebra morphism into a $C^*$\=/algebra $\A$. But we already know that there are other $^*$\=/representations
  that do not vanish on $Z$, like $\Phi_\lambda$ for $\lambda \neq 0$ composed with the Schrödinger representation
  of $\Weyl{1}$.
  Therefore any real Nullstellensatz for $\UnivStar(\lie h)$ must take into account $^*$\=/representations by unbounded operators.
\end{example}

\section{The real Nullstellensatz for \texorpdfstring{\boldmath$2$-step}{2-step} nilpotent Lie algebras}
\label{sec:theorem}

Recall the definitions of the set $\variety[\bullet]$ of filtered $^*$\=/homomorphisms from $\UnivStar(\lie g)$ to a Weyl algebra
(Definition~\ref{definition:varietyBullet}) and of the vanishing ideal $\Vanish{S}$ of $S\subseteq \variety[\bullet]$
(Definition~\ref{definition:vanishingIdeal}).

\begin{definition}
  Assume $\lie g$ is a finite-dimensional $2$-step nilpotent real Lie algebra and let $I$ be an ideal of the $^*$\=/algebra $\UnivStar(\lie g)$.
  We write
  \begin{equation}
    \Zeros{I} \coloneqq \set[\big]{ \Phi \in \variety[\bullet]}{I \subseteq \ker \Phi}
    .
  \end{equation}
\end{definition}
This set $\Zeros{I}$ will replace the variety of $^*$\=/characters $\ZerosNull{I}$ appearing in the commutative real
Nullstellensatz Lemma~\ref{lemma:commutativeNullstellensatz}; in contrast to $\ZerosNull{I}$, it contains also $^*$\=/homomorphisms
with values in Weyl algebras. Our main result is the following real Nullstellensatz for $2$-step nilpotent Lie algebras:

\begin{theorem} \label{theorem}
  Assume $\lie g$ is a finite-dimensional $2$-step nilpotent real Lie algebra and let $I$ be an ideal of the $^*$\=/algebra $\UnivStar(\lie g)$.
  Then the following are equivalent:
  \begin{enumerate}
    \item\label{item:theorem:real}
      $I$ is real.
    \item\label{item:theorem:Nullstellensatz}
      $I$ is the vanishing ideal of its set of zeros, i.e.~$I = \Vanish[\big]{\Zeros{I}}$.
  \end{enumerate}
\end{theorem}

\begin{startofproof}
  The implication \refitem{item:theorem:Nullstellensatz}~$\implies$~\refitem{item:theorem:real} has already been shown in Proposition~\ref{proposition:vanishingIsReal}.
  For the converse implication \refitem{item:theorem:real}~$\implies$~\refitem{item:theorem:Nullstellensatz} we first note
  that every real ideal $I$ of $\UnivStar(\lie g)$ clearly fulfils $I \subseteq \Vanish{\Zeros{I}}$. 
  It is not hard to see that, in order to prove the equality, it suffices to show that for every real ideal $I$ of $\UnivStar(\lie g)$
  there exists a subset $S$ of $\variety[\bullet]$ such that $I = \Vanish{S}$. Namely, the mappings $S \mapsto \Vanish{S}$ and
  $I \mapsto \Zeros{I}$ define an antitone Galois connection between the ideals of $\UnivStar(\lie g)$ and the subsets
  of $\variety[\bullet]$, therefore $\Vanish{S} = \Vanish{\Zeros{\Vanish{S}}}$. The construction of the set $S$ will occupy the rest of this section.
\end{startofproof}
In the next Section~\ref{sec:realIdealQuotients} we will develop an iterative process that allows us to split the construction
of such a set $S$ into easier subproblems. This process will then be implemented and will ultimately be completed in Section~\ref{sec:proofCompletion}.
Two instructive examples are discussed later in Section~\ref{sec:examples}.

\subsection{An iterative construction} \label{sec:realIdealQuotients}
In order to cope with e.g.~the different types of filtered $^*$\=/algebra morphisms $\Phi \colon \UnivStar(\lie g) \to \Weyl{d}$
for different $d\in \NN_0$ we will have to apply an iterative construction.
The important ingredient on the algebraic side to make this work are ideal quotients:

For any ring $\ring R$ and two ideals $I, J$ of $\ring R$, the \emph{ideal quotient of $I$ by $J$} is usually defined as
\begin{equation}
  I : J \coloneqq \set{r\in\ring R}{r J \subseteq I \textup{ and } J r \subseteq I}
  .
\end{equation}
It is easy to check that $I : J$ is again an ideal of $\ring R$ and that $I \subseteq I : J$.
The general idea behind this construction is that, at least in well-behaved commutative cases and if $I$ and $J$ are vanishing ideals
of some varieties $S$ and $T$, respectively, then $I:J$ is the vanishing ideal of $S \setminus T$.
We are especially interested in the case that $I$ is a real ideal of a not necessarily commutative $^*$\=/algebra:

\begin{lemma} \label{lemma:realIdealQuotients}
  Assume $\A$ is a $^*$-algebra, $I$ is a real ideal of $\A$ and $J$ a $^*$\=/ideal of $\A$.
  Then $I : J$ is again a real ideal of $\A$ and the identities
  \begin{equation} \label{eq:realIdealQuotients}
    I : J = \set[\big]{a\in\A}{a J \subseteq I}
    \quad\quad\text{and}\quad\quad
    (I : J) \cap J = I \cap J
  \end{equation}
  hold.
\end{lemma}
\begin{proof}
  It is clear that $I : J \subseteq \set{a\in\A}{a J \subseteq I}$. Conversely, consider an element $a\in \A$ that fulfils $a J \subseteq I$.
  Given $b \in J$, then $(ba)^* \in J$ because $J$ is a $^*$\=/ideal, so $a(ba)^* \in I$ by assumption, hence also $(ba)(ba)^* \in I$ and
  therefore $(ba)^* \in I$ because $I$ is real. As the real ideal $I$ is automatically a $^*$\=/ideal, this shows that $ba\in I$.
  So $J a \subseteq I$ holds, proving the first identity in \eqref{eq:realIdealQuotients}.

  Next we show that $I : J$ is real, so consider $a_1, \dots, a_k \in \A$, $k\in \NN$ such that $\sum_{j=1}^k a^*_j a_j \in I : J$.
  For $b \in J$ this means $\sum_{j=1}^k a^*_j a_j b \in I$, hence also $\sum_{j=1}^k (a_j b)^* (a_j b) \in I$.
  So $a_j b \in I$ for all $j \in \{1,\dots,k\}$ because $I$ is real.
  As $b \in J$ was arbitrary this shows $a_j \in I:J$ for all $j \in \{1,\dots,k\}$ by the first part.
  
  Finally, $I \cap J \subseteq (I:J) \cap J$ is clear since $I \subseteq I:J$.
  Conversely, consider $a \in (I:J) \cap J$. Then $a^* a \in I$ because $a \in I : J$ and $a^* \in J$.
  As $I$ is real it follows that $a \in I$.
  This shows that $(I:J) \cap J \subseteq I$, therefore also $(I:J) \cap J \subseteq I \cap J$,
  and thus the second identity in \eqref{eq:realIdealQuotients} is proven.
\end{proof}

Recall that a \emph{partially ordered set} is a tuple $(S,\le)$ of a set $S$ and a reflexive, transitive, and antisymmetric relation $\le$ on $S$.
Given such a partially ordered set $(S,\le)$ and two elements $s,t \in S$, then we write ``$s < t$'' for ``$s \le t$ and $s\neq t$'' as usual.

\begin{lemma} \label{lemma:partial2totalOrdering}
  Let $(\FinSet, \le)$ be a finite partially ordered set and write $M \in \NN_0$ for the number of elements of $\FinSet$.
  Then there is a bijection $\kappa \colon \{1,\dots,M\} \to \FinSet$, $m \mapsto \kappa_m$ such that $m < n$ holds for all
  $m,n \in \{1,\dots,M\}$ with $\kappa_m < \kappa_n$.
\end{lemma}
\begin{proof}
  Recall that every nonempty finite partially ordered set $(S, \le)$ has a minimal element $s_{\min} \in S$,
  i.e.~$\set{s \in S}{s < s_{\min}} = \emptyset$ (note that the existence of a minimal element in general does not imply the existence of the minimum).
  Such a minimal element can e.g.~be obtained as follows:
  Set $k \coloneqq 0$ and choose any element $s_k\in S$; as long as $s_k$ is not minimal in $S$,
  choose $s_{k+1} \in \set{s \in S}{s < s_k}$ and repeat with $k+1 \in \NN_0$ in place of $k$.
  This iteration results in a strictly decreasing sequence $s_0 > \dots > s_k$,
  which must terminate after finitely many steps at a minimal element $s_k$ of $S$.
  
  We can thus recursively construct a bijection $\kappa \colon \{1,\dots,M\} \to \FinSet$ in such a way that
  $\kappa_m$ is minimal in $\FinSet \setminus \{\kappa_1, \dots, \kappa_{m-1}\}$ for all $m \in \{1,\dots,M\}$.
  Now consider $m,n \in \{1,\dots,M\}$ with $\kappa_m < \kappa_n$. Then $\kappa_m \in \{\kappa_1, \dots, \kappa_{n-1}\}$
  by minimality of $\kappa_n$, so $m \in \{1,\dots,n-1\}$ because $\kappa$ is a bijection.
\end{proof}

\begin{proposition} \label{proposition:stepwise}
  Consider a $^*$\=/algebra $\A$, a finite partially ordered set $(\FinSet, \le)$,
  a map $\FinSetMapMap \colon \FinSet \to \A$, $K \mapsto \FinSetMap{K}$, and assume that
  the minimum $\min \FinSet$ exists and that $\FinSetMap{\min \FinSet} = \Unit$.
  An ideal $I$ of $\A$ is said to be \emph{of type $K$} for $K \in \FinSet$ if $\FinSetMap{L} \in I$ for all $L \in \FinSet$ with $K < L$.
  Assume that to every tuple $(K,I)$ with $K \in \FinSet$ and $I$ a real ideal of $\A$ of type $K$
  there is assigned a set $S(K,I)$ of $^*$\=/algebra morphisms with domain $\A$ such that
  \begin{equation}
    \label{eq:stepwise}
    \FinSetMap{K} \, \Vanish[\big]{S(K,I)} \subseteq I \subseteq \Vanish[\big]{S(K,I)}
  \end{equation}
  holds.
  Then for every real ideal $I$ of $\A$ there are $M\in \NN$, elements $\kappa_1, \dots, \kappa_M \in \FinSet$,
  and real ideals $J_1, \dots, J_M \subseteq \A$ with $J_m$ of type $\kappa_m$
  for all $m\in \{1,\dots,M\}$ such that $\Vanish[\big]{\bigcup_{m=1}^M S(\kappa_m,J_m)} = I$.
\end{proposition}
\begin{proof}
  Write $M$ for the number of elements in $\FinSet$ and let $\kappa \colon \{1,\dots,M\} \to \FinSet$, $m\mapsto \kappa_m$ be a bijection as
  in the previous Lemma~\ref{lemma:partial2totalOrdering}, i.e.~whenever $\kappa_m<\kappa_n$ for $m,n\in \{1,\dots,M\}$, then $m<n$.
  Note that this implies $\min \FinSet = \kappa_1$.
  Let $I$ be any real ideal of $\A$.
  For all $m\in \{1,\dots,M\}$ we recursively construct a real ideal $J_m$ of $\A$ in such a way that 
  $\FinSetMap{{\kappa_n}} \in J_m$ for all $n \in \{m+1,\dots,M\}$, so in particular $J_m$ is of type $\kappa_m$:
  
  First note that $M\ge 1$ because $\min \FinSet \in \FinSet$ exists by assumption, and so we define $J_{\kappa(M)} \coloneqq I$.
  Now assume that $J_m$ has been defined for one $m\in \{1,\dots,M\}$ and that
  $\FinSetMap{{\kappa_n}} \in J_m$ holds for all $n \in \{m+1,\dots,M\}$ as required above, so the set $S(\kappa_m,J_m)$ is defined.
  The ideal quotient $J_m : \Vanish{S({\kappa_m},J_m)}$ is a real ideal by Lemma~\ref{lemma:realIdealQuotients},
  and $\FinSetMap{\kappa_n} \in J_m \subseteq J_m : \Vanish{S({\kappa_m},J_m)}$ for all $n \in \{m+1,\dots,M\}$.
  Moreover, $\FinSetMap{\kappa_m}\,\Vanish{S({\kappa_m},J_m)} \subseteq J_m$ by assumption,
  and therefore $\FinSetMap{\kappa_m} \in J_m : \Vanish{S({\kappa_m},J_m)}$ by Lemma~\ref{lemma:realIdealQuotients}.
  So as long as $m>1$ we can define $J_{m-1} \coloneqq J_m : \Vanish{S({\kappa_m},J_m)}$
  and continue the construction with $m-1$ in place of $m$.
  However, when finally $m=1$, then $\kappa_1 = \min \FinSet$, so $\FinSetMap{\kappa_1} = \FinSetMap{\min \FinSet} = \Unit$ by assumption and thus 
  $J_1 = \Vanish{S(\kappa_1,J_1)}$ by \eqref{eq:stepwise}.
  
  We will now show that $\Vanish[\big]{\bigcup_{m=1}^n S(\kappa_m,J_m)} = J_n$ for all $n\in \{1,\dots,M\}$:
  For $n = 1$ we have just seen that
  $\Vanish[\big]{\bigcup_{m=1}^1 S(\kappa_m,J_m)} = \Vanish{S(\kappa_1,J_1)} = J_1$.
  Now assume that $\Vanish[\big]{\bigcup_{m=1}^n S(\kappa_m,J_m)} = J_n$ holds for some $n \in \{1,\dots,M-1\}$.
  Then also
  \begin{align*}
    J_{n+1}
    &=
    J_{n+1} \cap \Vanish[\big]{S({\kappa_{n+1}}, J_{n+1})}
    \\
    &=
    \bigl(J_{n+1} : \Vanish[\big]{S({\kappa_{n+1}}, J_{n+1})}\bigr) \cap \Vanish[\big]{S({\kappa_{n+1}}, J_{n+1})}
    \\
    &=
    J_n \cap \Vanish[\big]{S({\kappa_{n+1}}, J_{n+1})}
    \\
    &=
    \Vanish[\Big]{\bigcup\nolimits_{m=1}^{n+1} S(\kappa_m,J_m)}
  \end{align*}
  holds,
  where we first use that $J_{n+1} \subseteq \Vanish[\big]{S({\kappa_{n+1}}, J_{n+1})}$ by \eqref{eq:stepwise},
  then we apply Lemma~\ref{lemma:realIdealQuotients} and then the definition of $J_n$;
  the last equality holds by the induction assumption $J_n = \Vanish[\big]{\bigcup_{m=1}^n S(\kappa_m,J_m)}$.
  By induction it follows that $\Vanish[\big]{\bigcup_{m=1}^n S(\kappa_m,J_m)} = J_n$ for all $n\in \{1,\dots,M\}$,
  and in particular $\Vanish[\big]{\bigcup_{m=1}^M S(\kappa_m,J_m)} = J_M = I$.
\end{proof}
As discussed in the first part of the proof of Theorem~\ref{theorem}, we have to construct, for every real ideal $I$ of $\UnivStar(\lie g)$,
a subset $S(I)$ of $\variety[\bullet]$ such that $\Vanish{S(I)} = I$. The above Proposition~\ref{proposition:stepwise}
allows us to split this problem into smaller and easier subproblems \eqref{eq:stepwise}, each of which will eventually be solved by application of the classical,
commutative real Nullstellensatz.
Heuristically, the assumption in \eqref{eq:stepwise} that $I$ is of type $K$ means that such a subproblem
is concerned only with the quotient of $\UnivStar(\lie g)$ modulo the smallest real ideal of type $K$,
which corresponds to the ``Zariski closed'' subset $\Zeros[\big]{\set{\FinSetMap{L}}{L \in \FinSet\text{ with } K < L}}$ of $\variety[\bullet]$.
The relaxed condition $\FinSetMap{K} \, \Vanish{S(K,I)} \subseteq I \subseteq \Vanish{S(K,I)}$
then means that such a subproblem has to be solved only in a localization at $\FinSetMap{K}$, i.e.~a further restriction
to the ``relatively Zariski open'' subset $\Zeros[\big]{\set{\FinSetMap{L}}{L \in \FinSet\text{ with } K < L}} \setminus \Zeros[\big]{\{\FinSetMap{K}\}}$.

\subsection{Pfaffians}

Throughout the rest of Section~\ref{sec:theorem} we fix a finite-dimensional $2$-step nilpotent real Lie algebra $\lie g$ and we write 
\begin{equation}
  \lie z \coloneqq \set[\big]{Z \in \lie g}{ \kom{Z}{X} = 0 \text{ for all }X\in \lie g}
\end{equation}
for its center. We choose a basis $C_1, \dots, C_\gamma$ with $\gamma \in \NN_0$ of the center $\lie z$ and extend this
to a basis $B_1, \dots, B_\beta, C_1, \dots, C_\gamma$ with $\beta \in \NN_0$ of whole $\lie g$.

In order to apply the iterative construction from Proposition~\ref{proposition:stepwise} we have to specify a finite partially
ordered set $\FinSet$ and a suitable map $\FinSetMapMap \colon \FinSet \to \UnivStar(\lie g)$. This map will be given by the Pfaffians
of matrices that are built from commutators of basis elements $B_1, \dots, B_\beta$ (see e.g.~\cite[Sec.~7]{godsil:algebraicCombinatorics}):

Recall that the determinant of any antisymmetric matrix $A \in \ring{R}^{2d\times 2d}$, $d\in \NN_0$ over a commutative ring $\ring{R}$
is the square of its Pfaffian, $\det(A) = \pfaffian(A)^2$. This Pfaffian is defined as
\begin{equation}
  \pfaffian(A) \coloneqq \frac{1}{2^d d!} \sum_{\sigma\in\SymmetricGroup_{2d}} \sign(\sigma) \prod_{j=1}^d A_{\sigma(2j-1), \sigma(2j)}
\end{equation}
where $\SymmetricGroup_{2d}$ is the group of permutations of $\{1,\dots,2d\}$ and $\sign \colon \SymmetricGroup_{2d} \to \{-1,+1\}$
describes the sign of a permutation.
Analogously to the Laplace expansion formula, Pfaffians fulfil the identity
\begin{equation}
  \label{eq:pfaffianMatrixExpansion}
  \pfaffian(A) = \sum_{j=1}^{2d-1} (-1)^{j+1} A_{j,2d} \pfaffian\bigl(\widehat{A}_{j,2d}\bigr) \quad\quad\text{if $d>0$}
  ,
\end{equation}
where the antisymmetric matrix $\widehat{A}_{j,2d} \in \ring{R}^{2(d-1)\times 2(d-1)}$ is obtained from the matrix $A$ 
by removing the rows and columns indexed by $j$ and $2d$.

Note that $\genSAlg{\lie z}$, the $^*$\=/subalgebra of $\UnivStar(\lie g)$ that is generated by the center $\lie z$ of $\lie g$,
is a commutative ring.
\begin{definition}
  For $d \in \NN_0$ we define the map $\Pfaffian \colon \lie g^{2d} \to \genSAlg{\lie z}$,
  \begin{equation}
    \label{eq:Pfaffian}
    X_1, \dots, X_{2d}
    \mapsto
    \Pfaffian(X_1,\dots,X_{2d})
    \coloneqq
    \pfaffian\Bigl( \bigl(\I \kom{X_k}{X_\ell}\bigr)_{k,\ell=1}^{2d} \Bigr)
    .
  \end{equation}
  We write $\FinSet$ for the set of all (possibly empty) subsets of $\{1,\dots,\beta\}$ that contain an even number of elements
  and endow $\FinSet$ with the partial order $\subseteq$. For elements $K = \{k_1, \dots, k_{2d}\} \in \FinSet$ 
  with $k_1 < \dots < k_{2d}$ and $d\in \NN_0$ we define the shorthand
  \begin{equation}
    \Pfaffian(K) \coloneqq \Pfaffian\bigl(B_{k_1}, \dots, B_{k_{2d}}\bigr)  \in \genSAlg{\lie z}
    .
  \end{equation}
\end{definition}
The factor $\I$ in \eqref{eq:Pfaffian} assures that $\Pfaffian(X_1,\dots,X_{2d})$
is hermitian for all $X_1,\dots,X_{2d}\in \lie g$, $d\in \NN_0$, because the matrix
$(\I \kom{X_k}{X_\ell})_{k,\ell=1}^{2d}$ has only hermitian entries.
Note that $\FinSet$ contains only sets with an even number of elements. The reason for this is that the Pfaffian and determinant
of an odd-dimensional antisymmetric matrix are zero. The minimum of this partially ordered set $(\FinSet, \subseteq)$ clearly
exists, namely $\min \FinSet = \emptyset$, and it will be important for the proof of Theorem~\ref{theorem} that all of the
following also holds in the trivial case $K = \emptyset \in \FinSet$ where $\Pfaffian(\emptyset) = \Unit$,
with the usual definition that empty sums are zero and empty products are the unit.

It is easy to check that the map $\Pfaffian \colon \lie g^{2d} \to \genSAlg{\lie z}$ is multilinear and that 
\begin{equation}
  \Pfaffian\bigl(X_{\tau(1)},\dots,X_{\tau(2d)}\bigr) = \sign(\tau) \Pfaffian(X_1,\dots,X_{2d})
\end{equation}
for all $\tau \in \SymmetricGroup_{2d}$, $X_1,\dots,X_{2d} \in \lie g$, $d\in \NN$, so $\Pfaffian$ is an alternating multilinear map.
Moreover, the expansion formula \eqref{eq:pfaffianMatrixExpansion} yields
\begin{equation}
  \label{eq:PfaffianExpansion}
  \Pfaffian(X_1, \dots, X_{2d})
  =
  \I
  \sum_{j=1}^{2d-1}
  (-1)^{j+1}
  \Pfaffian\bigl(X_1, \dots, X_{j-1}, X_{j+1}, \dots, X_{2d-1}\bigr)
  \kom{X_j}{X_{2d}}
\end{equation}
for all $X_1,\dots,X_{2d} \in \lie g$, $d\in \NN$.
In order to keep later formulas compact we introduce the following notation:

\begin{definition}
  For any $S \subseteq \NN$ and $k\in \NN$ we write $\sign_S(k) \coloneqq (-1)^{c(S,k)}$
  with $c(S,k)$ defined as the number of elements in $S$ that are strictly smaller than $k$, i.e.~$c(S,k) \coloneqq \sum_{s \in S,s < k}1$.
\end{definition}

\begin{definition}
  Let $K'$ be a subset of $\{1,\dots,\beta\}$ with an odd number of elements, then we define
  \begin{align}
    \Pfaffian^k(K')
    &\coloneqq
    \sign_{K'}(k)\Pfaffian\bigl(K'\setminus\{k\}\bigr)
  \intertext{for $k \in K'$, and}
    \Pfaffian_h(K')
    &\coloneqq
    \begin{cases}
      0 &\text{if }h\in K' \\
      -\sign_{K'}(h)\Pfaffian\bigl(K' \cup \{h\}\bigr) &\text{if }h\notin K'
    \end{cases}
  \end{align}
  for $h \in \{1,\dots,\beta\}$.
\end{definition}

\begin{proposition} \label{proposition:PfaffianMatrixVectorOdd}
  Let $K'$ be a subset of $\{1,\dots,\beta\}$ with an odd number of elements. Then the identity
  \begin{equation} \label{eq:PfaffianMatrixVectorOdd}
    {\sum_{k \in K'} \Pfaffian^k(K') \kom{B_k}{B_h}}
    =
    -\I \Pfaffian_h(K')
  \end{equation}
  holds for all $h \in \{1,\dots,\beta\}$.
\end{proposition}
\begin{proof}
  Write $\{k_1, \dots, k_{2d-1}\} \coloneqq K'$ with $d\in \NN$ and $k_1 < \dots < k_{2d-1}$. Then $\sign_{K'}(k_j) = (-1)^{j+1}$
  for all $j\in \{1,\dots,2d-1\}$, and $\Pfaffian_h(K') = \Pfaffian(B_{k_1}, \dots, B_{k_{2d-1}}, B_h)$ for all $h\in \{1,\dots,\beta\}$
  because $\Pfaffian \colon \lie g^{2d} \to \genSAlg{\lie z}$ is alternating. So \eqref{eq:PfaffianMatrixVectorOdd}
  is just a reformulation of the expansion formula \eqref{eq:PfaffianExpansion}.
\end{proof}

\subsection{Invariant and dual elements} \label{sec:invariantAndDual}

We can now construct, for every $K \in \FinSet$, two special types of well-behaved elements of $\UnivStar(\lie g)$,
namely the \emph{invariant} elements $\Invariant_{K,m}$ for $m\in \{1,\dots,\beta\} \setminus K$ and the \emph{dual} elements $\Dual_{K,\ell}$
for $\ell \in K$:

\begin{definition} \label{definition:generator}
  For $K \in \FinSet$ and $m \in \{1,\dots, \beta\} \setminus K$ we define the element $\Invariant_{K,m} \in \UnivStar(\lie g)$ as
  \begin{align}
    \label{eq:generator:gen}
    \Invariant_{K,m}
    &\coloneqq
    \sign_K(m)\sum_{k\in K \cup \{m\}} \Pfaffian^k\bigl(K\cup \{m\}\bigr)\, B_k
    .
  \intertext{Similarly, for $K \in \FinSet$ and $\ell\in K$ we define the element $\Dual_{K,\ell} \in \UnivStar(\lie g)$ as}
    \Dual_{K,\ell}
    &\coloneqq
    \sign_K(\ell)\sum_{k \in K \setminus \{\ell\}}
    \Pfaffian^k\bigl(K\setminus \{\ell\}\bigr)\,B_k
    .
  \end{align}
\end{definition}
Note that these elements $\Invariant_{K,m}$ and $\Dual_{K,\ell}$ of $\UnivStar(\lie g)$ in general are not elements of $\lie g$,
but they are again antihermitian. We will often use the equivalent formula
\begin{equation}
  \Invariant_{K,m}
  =
  \Pfaffian(K) \,B_k
  +
  \sign_K(m)\sum_{k\in K} \Pfaffian^k\bigl(K\cup \{m\}\bigr)\, B_k
\end{equation}
for $K \in \FinSet$ and $m \in \{1,\dots, \beta\} \setminus K$.

\begin{proposition} \label{proposition:commutators}
  Assume $K \in \FinSet$. Then
  \begin{align}
    \label{eq:commutators:invariant}
    {\kom[\big]{\Invariant_{K,m}}{B_h}} &= -\I \sign_K(m) \Pfaffian_h\bigl(K\cup \{m\}\bigr)
  \intertext{for all $m\in \{1,\dots,\beta\} \setminus K$ and all $h\in \{1,\dots,\beta\}$, and}
    \label{eq:commutators:mixed}
    {\kom[\big]{\Invariant_{K,m}}{\Dual_{K,\ell}}} &= 0
  \intertext{for all $m\in \{1,\dots,\beta\} \setminus K$ and all $\ell \in K$. Moreover,}
    \label{eq:commutators:dual}
    {\kom[\big]{\Dual_{K,\ell}}{B_{\ell'}}} &= \I \Kronecker_{\ell,\ell'} \Pfaffian(K)
  \end{align}
  for all $\ell,\ell' \in K$, where $\Kronecker_{\ell,\ell'}$ is the Kronecker-$\delta$.
\end{proposition}
\begin{proof}
  Given $m\in \{1,\dots,\beta\} \setminus K$ and $h\in \{1,\dots,\beta\}$, then
  \begin{equation*}
    \kom[\big]{\Invariant_{K,m}}{B_h}
    =
    \sign_K(m) \sum_{k\in K \cup \{m\}} \Pfaffian^k\bigl(K\cup \{m\}\bigr) \kom{B_k}{B_h}
    =
    -\I \sign_K(m) \Pfaffian_h\bigl(K\cup \{m\}\bigr)
  \end{equation*}
  where we first use that $\Pfaffian^k(K\cup \{m\})$ is central in $\UnivStar(\lie g)$ and then we apply Proposition~\ref{proposition:PfaffianMatrixVectorOdd}.
  For $h \coloneqq k\in K$ this shows that $\kom[\big]{\Invariant_{K,m}}{B_k} = 0$ because $\Pfaffian_k\bigl(K\cup \{m\}\bigr) = 0$.
  It follows that $\kom[\big]{\Invariant_{K,m}}{\Dual_{K,\ell}} = 0$ for all $m\in \{1,\dots,\beta\} \setminus K$, $\ell \in K$
  because the prefactors $\Pfaffian^k(K\setminus \{\ell\})$ appearing in $\Dual_{K,\ell}$ are central anyways.
  A similar argument shows that, for arbitrary $\ell,\ell' \in K$,
  \begin{equation*}
    \kom[\big]{\Dual_{K,\ell}}{B_{\ell'}}
    =
    \sign_K(\ell)\sum_{k \in K \setminus \{\ell\}} \Pfaffian^k\bigl(K\setminus \{\ell\}\bigr) \kom{B_k}{B_{\ell'}}
    =
    -\I \sign_K(\ell) \Pfaffian_{\ell'}\bigl(K\setminus \{\ell\}\bigr)
    =
    \I \Kronecker_{\ell,\ell'} \Pfaffian(K)
    .
  \end{equation*}  
  Here we use in the last step for $\ell\neq \ell'$ that $\Pfaffian_{\ell'}\bigl(K\setminus \{\ell\}\bigr) = 0$,
  and if $\ell= \ell'$, then $\Pfaffian_{\ell'}\bigl(K\setminus \{\ell\}\bigr) = -\sign_{K \setminus\{\ell\}}(\ell') \Pfaffian(K)
  = -\sign_K(\ell) \Pfaffian(K)$.
\end{proof}
For $K \in \FinSet$, $m\in \{1,\dots,\beta\} \setminus K$ and $h \in \{1,\dots,\beta\}$
identity \eqref{eq:commutators:invariant} shows that $\kom[\big]{\Invariant_{K,m}}{B_h}$
is contained in every ideal $I$ of $\UnivStar(\lie g)$ that fulfils $\Pfaffian(L) \in I$
for all $L \in \FinSet$ with $K \subsetneq L$. It is in this way that $\Invariant_{K,m}$ is ``invariant''.

\subsection{An expansion over invariant elements}

\begin{definition} \label{definition:UnivInv}
  Assume $K \in \FinSet$, then we define $\UnivStar(\lie g)_\subKnachKlammer^\inv$ as the unital
  $^*$\=/subalgebra of $\UnivStar(\lie g)$ that is generated by all the elements $\Invariant_{K, m}$ with $m \in \{1,\dots, \beta\} \setminus K$
  together with the center of $\lie g$, i.e.
  \begin{equation}
    \label{eq:generator:0}
    \UnivStar(\lie g)_\subKnachKlammer^\inv \coloneqq \genSAlg[\big]{\set[\big]{\Invariant_{K,m}}{m \in \{1,\dots, \beta\} \setminus K} \cup \{C_1, \dots, C_\gamma\}}
    .
  \end{equation}
\end{definition}
These algebras $\UnivStar(\lie g)_K^\inv$ are the main technical object of the paper.
They enable the passage from the commutative real Nullstellensatz (Lemma~\ref{lemma:commutativeNullstellensatz}) to the
noncommutative real Nullstellenatz (Theorem~\ref{theorem}): Lemma~\ref{lemma:commutativeNullstellensatz}
applies to the restriction $\UnivStar(\lie g)_K^\inv \cap I$ of certain well-behaved real ideals $I$, and in 
Proposition~\ref{proposition:constructionOfRepFromFunctional} we will show that many of the resulting $^*$\=/characters
of $\UnivStar(\lie g)_K^\inv$ can be extended to filtered $^*$\=/homomorphisms from $\UnivStar(\lie g)$ to a Weyl algebra, whose kernels contain $I$.
Heuristically, the former requires that $\UnivStar(\lie g)_K^\inv$ has to be sufficiently small so that it does not
contain elements whose commutators lie outside of $I$, while the latter means that $\UnivStar(\lie g)_K^\inv$ has to be
sufficiently large so that $\UnivStar(\lie g)_K^\inv \cap I$ retains enough information about $I$.
More precisely, in Propositions~\ref{proposition:commutatortrick} we will show that all elements of any ideal of $\UnivStar(\lie g)$ essentially
allow an expansion as a finite series of powers in $B_k$, $k\in K$ with prefactors from $\UnivStar(\lie g)_\subKnachKlammer^\inv$.
An instructive example of $\UnivStar(\lie g)^\inv_K$ will be given in Section~\ref{sec:examples:nonparallel}.

\begin{definition}
  We introduce a multiindex notation for powers in $\UnivStar(\lie g)$ of the basis elements $B_k$, $k\in K$:
  Consider $K = \{k_1, \dots, k_{2d}\} \in \FinSet$ with $d\in \NN_0$ and $k_1 < \dots < k_{2d}$ and any multiindex 
  $\alpha \in \NN_0^{2d}$, then we define
  \begin{equation}
    B_\subKnachKlammer^\alpha \coloneqq (B_{k_1})^{\alpha_1} \dots (B_{k_{2d}})^{\alpha_{2d}} \in \UnivStar(\lie g)
    .
  \end{equation}
  For any such multiindex $\alpha \in \NN_0^{2d}$ we also set
  \begin{equation}
    \abs{\alpha} \coloneqq \sum_{j=1}^{2d} \alpha_j \in \NN_0
    \quad\quad\text{and}\quad\quad
    \alpha! \coloneqq \prod_{j=1}^{2d} \alpha_j! \in \NN
  \end{equation}
  as usual, and we define the elementwise partial order on multiindices,
  i.e.~$\alpha \le \alpha'$ for $\alpha,\alpha' \in \NN_0^{2d}$ means that $\alpha_j \le \alpha_j'$ for all $j \in \{1,\dots,2d\}$.
\end{definition}

\begin{lemma} \label{lemma:dual}
  Assume $K = \{k_1,\dots,k_{2d}\} \in \FinSet$ with $d \in \NN_0$ and $k_1< \dots < k_{2d}$.
  For $\alpha \in \NN_0^{2d}$ define
  \begin{equation}
    \ad_D^{\alpha} \coloneqq \bigl(\ad_{D,1}\bigr)^{\alpha_1} \dots \bigl(\ad_{D,2d}\bigr)^{\alpha_{2d}} \colon \UnivStar(\lie g) \to \UnivStar(\lie g)
  \end{equation}
  where $\ad_{D,j}$ with $j \in \{1,\dots,2d\}$ is the inner derivation $\ad_{D,j} \colon \UnivStar(\lie g) \to \UnivStar(\lie g)$,
  \begin{equation}
    a \mapsto \ad_{D,j}(a) \coloneqq {\kom[\big]{\Dual_{K,k_j}}{a}}
    .
  \end{equation}
  Then the iterated inner derivations $\ad_D^{\alpha'}$ with ${\alpha'} \in \NN_0^{2d}$ fulfil the equations
  \begin{alignat}{2}
    \label{eq:dual:inv}
    \ad_D^{{\alpha'}}(p) &= 0 &&\text{for all $p\in \UnivStar(\lie g)_\subKnachKlammer^\inv$}
  \shortintertext{and}
    \label{eq:dual:B}
    \ad_D^{{\alpha'}}\bigl( B_\subKnachKlammer^{\alpha}\bigr)
    &=
    \begin{cases}
      \frac{\left(\I \Pfaffian(K)\right)^{\abs{{\alpha'}}} \alpha! }{(\alpha-{\alpha'})!} B_\subKnachKlammer^{\alpha-{\alpha'}} & \text{if ${\alpha'} \le \alpha$} \\
      0 & \text{otherwise}
    \end{cases}
    &\quad\quad\quad&\text{for all $\alpha \in \NN_0^{2d}$.}
  \end{alignat}
\end{lemma}
\begin{proof}
  Consider $j \in \{1,\dots,2d\}$, then $\ad_{D,j}(\Invariant_{K,m}) = {\kom{\Dual_{K,k_j}}{\Invariant_{K,m}}} = 0$ for all $m\in \{1,\dots,\beta\} \setminus K$
  by identity \eqref{eq:commutators:mixed} of Proposition~\ref{proposition:commutators}, and $\ad_{D,j}(C_{j'}) = 0$ for all $j' \in \{1,\dots,\gamma\}$ because $C_{j'}$
  is central. This shows that equation~\eqref{eq:dual:inv} is fulfilled for all ${\alpha'} \in \NN_0^{2d}$ and all $p\in \UnivStar(\lie g)_\subKnachKlammer^\inv$.
  
  Now write $e_i \in \NN_0^{2d}$ with $i \in \{1,\dots,2d\}$ for the multiindex given by $(e_i)_i \coloneqq 1$ and $(e_i)_{i'} \coloneqq 0$ for $i' \in \{1,\dots,2d\} \setminus \{i\}$.
  For $j\in \{1,\dots,2d\}$ identity \eqref{eq:commutators:dual} of Proposition~\ref{proposition:commutators} shows that
  $\ad_{D,j}\bigl(B_\subKnachKlammer^{\alpha}\bigr) = \I \alpha_j \Pfaffian(K) B_\subKnachKlammer^{\alpha-e_j}$
  for all $\alpha \in \NN_0^{2d}$ with $\alpha_j>0$, and similarly, $\ad_{D,j}\bigl(B_\subKnachKlammer^{\alpha}\bigr) = 0$
  for all $\alpha \in \NN_0^{2d}$ with $\alpha_j=0$. From this it follows that equation~\eqref{eq:dual:B} is fulfilled for all
  ${\alpha'},\alpha \in \NN_0^{2d}$ because $\prod_{i = 0}^{\alpha'_j-1} (\alpha_j-i) = \alpha_j! \quotient (\alpha_j- \alpha'_j)!$
  holds for all ${\alpha'},\alpha \in \NN_0^{2d}$ with ${\alpha'} \le \alpha$ and all $j\in \{1,\dots,2d\}$.
\end{proof}

\begin{lemma} \label{lemma:univstarinfty}
  Assume $K = \{k_1,\dots,k_{2d}\} \in \FinSet$ with $d \in \NN_0$ and $k_1< \dots < k_{2d}$.
  Then for all $a \in \UnivStar(\lie g)$ there are $r,s \in \NN_0$ and coefficients $p_\alpha \in \UnivStar(\lie g)_\subKnachKlammer^\inv$
  for $\alpha \in \NN_0^{2d}$ with $\abs{\alpha} \le s$ such that
  \begin{equation}
    \Pfaffian(K)^r a = \sum_{\alpha \in \NN_0^{2d}, \abs{\alpha} \le s} p_\alpha B_\subKnachKlammer^\alpha .
  \end{equation}
\end{lemma}
The number $s$ appearing in this formula guarantees that the sum is finite.
\begin{proof}
  For $s \in \NN_0$ we write
  \begin{equation*}
    \UnivStar(\lie g)_\subKnachKlammer^s
    \coloneqq
    \set[\Big]{
      \sum\nolimits_{\alpha \in \NN_0^{2d}, \abs{\alpha} \le s} p_\alpha B_\subKnachKlammer^\alpha
    }{
      p_\alpha \in \UnivStar(\lie g)_\subKnachKlammer^\inv\textup{ for all }\alpha \in \NN_0^{2d}\textup{ with }\abs{\alpha} \le s
    }
    .
  \end{equation*}
  Clearly $\UnivStar(\lie g)_\subKnachKlammer^0 = \UnivStar(\lie g)_\subKnachKlammer^\inv$.
  Consider any $s \in \NN_0$ and an arbitrary element $a = \sum_{\alpha \in \NN_0^{2d}, \abs{\alpha} \le s} p_\alpha B_\subKnachKlammer^\alpha \in \UnivStar(\lie g)_\subKnachKlammer^s$.
  Then $a q = \sum_{\alpha \in \NN_0^{2d}, \abs{\alpha} \le s} p_\alpha q B_\subKnachKlammer^\alpha \in \UnivStar(\lie g)_\subKnachKlammer^s$ for all $q \in \UnivStar(\lie g)_\subKnachKlammer^\inv$ because
  $\kom{B_{k_j}}{C_{j'}} = 0$ for all $j,j' \in \{1,\dots,2d\}$ anyways and because
  $\kom{B_{k_j}}{\Invariant_{K,m}} = \I \sign_K(m) \Pfaffian_{k_j}\bigl(K\cup \{m\}\bigr) = 0$ for all $j\in \{1,\dots,2d\}$
  and all $m\in \{1,\dots, \beta\} \setminus K$ by Proposition~\ref{proposition:commutators}
  and Definition~\ref{definition:generator}.
  Moreover, $a B_{k_j} \in \UnivStar(\lie g)_\subKnachKlammer^{s+1}$ for all $j\in \{1,\dots,2d\}$:
  indeed, write $e_i \in \NN_0^{2d}$ with $i \in \{1,\dots,2d\}$ for the multiindex given by $(e_i)_i \coloneqq 1$ and 
  $(e_i)_{i'} \coloneqq 0$ for $i' \in \{1,\dots,2d\} \setminus \{i\}$, and consider any $\alpha \in \NN_0^{2d}$ with $\abs{\alpha} \le s$,
  then
  \begin{align*}
    p_\alpha B_\subKnachKlammer^\alpha B_{k_j}
    &=
    p_\alpha B_\subKnachKlammer^{\alpha+e_j} + p_\alpha (B_{k_1})^{\alpha_1} \dots (B_{k_j})^{\alpha_j} \kom[\big]{(B_{k_{j+1}})^{\alpha_{j+1}} \dots (B_{k_{2d}})^{\alpha_{2d}}}{B_{k_j}}
    \\
    &=
    p_\alpha B_{K}^{\alpha+e_j} + \sum_{\substack{i=j+1 \\ \alpha_i > 0}}^{2d} \alpha_i p_\alpha \kom{B_{k_i}}{B_{k_j}} B_{K}^{\alpha-e_i} 
    \in \UnivStar(\lie g)_\subKnachKlammer^{s+1}
  \end{align*}
  because $\kom{B_{k_i}}{B_{k_j}} \in \lie z \subseteq \UnivStar(\lie g)_\subKnachKlammer^\inv$ for all $i \in \{j+1,\dots,2d\}$.
  It follows that $ab \in \UnivStar(\lie g)_\subKnachKlammer^{s+1}$ for all $b\in \UnivStar(\lie g)_\subKnachKlammer^1$.
  
  Now consider any element $a \in \UnivStar(\lie g)$. In order to complete the proof we have to find $r,s\in \NN_0$ such that
  $\Pfaffian(K)^r a \in \UnivStar(\lie g)_\subKnachKlammer^s$. So note that the identity
  \begin{equation*}
    \Pfaffian(K)\,B_m
    =
    \sign_K(m) \Pfaffian^m\bigl( K \cup \{m\} \bigr)\, B_m
    =
    \Invariant_{K,m} - \sign_K(m) \sum_{k\in K} \Pfaffian^k\bigl(K\cup \{m\}\bigr) \,B_k
    \in
    \UnivStar(\lie g)_\subKnachKlammer^1
  \end{equation*}
  holds for all $m\in \{1,\dots,\beta\} \setminus K$ by the definitions of $\Invariant_{K,m}$ and because $\sign_K(m) \Pfaffian^m\bigl( K \cup \{m\} \bigr)
  = \sign_K(m) \sign_{K\cup \{m\}}(m)  \Pfaffian(K) = \Pfaffian(K)$.  
  The element $a \in \UnivStar(\lie g)$ can now be expanded as a linear combination of products of the basis elements
  $B_1, \dots, B_\beta, C_1,\dots C_\gamma$ of $\lie g$. Let $r \in \NN_0$ be the maximal number of factors $B_m$ with $m\in\{1,\dots, \beta\} \setminus K$
  that occur in any summand of this linear combination.
  As $\Pfaffian(K)$ is central, one can expand $\Pfaffian(K)^r a$ as a linear combination of products with factors
  $B_{k_i} \in \UnivStar(\lie g)_\subKnachKlammer^1$, $\Pfaffian(K) B_m \in \UnivStar(\lie g)_\subKnachKlammer^1$, and $C_j \in \UnivStar(\lie g)_\subKnachKlammer^\inv$
  with $i \in \{1,\dots,2d\}$, $m \in \{1,\dots,\beta\} \setminus K$, and $j\in \{1,\dots, \gamma\}$, and with possibly
  some remaining factors $\Pfaffian(K) \in \genSAlg{\lie z} \subseteq \UnivStar(\lie g)_\subKnachKlammer^\inv$.
  It follows that $\Pfaffian(K)^r a \in \UnivStar(\lie g)_\subKnachKlammer^s$ for sufficiently large $s\in \NN_0$.
\end{proof}

\begin{proposition} \label{proposition:commutatortrick}
  Assume $K = \{k_1,\dots,k_{2d}\} \in \FinSet$ with $d\in \NN_0$ and $k_1 < \dots < k_{2d}$ and let $I$ be an ideal of $\UnivStar(\lie g)$.
  Then for all $a\in I$ there are $r,s\in \NN_0$ and coefficients $p_\alpha \in \UnivStar(\lie g)_\subKnachKlammer^\inv \cap I$
  such that
  \begin{equation}
    \label{eq:commutatortrick}
    \Pfaffian(K)^r a = \sum_{\alpha \in \NN_0^{2d} , \abs{\alpha} \le s} p_\alpha B_\subKnachKlammer^\alpha
    .
  \end{equation}
\end{proposition}
\begin{proof}
  By the previous Lemma~\ref{lemma:univstarinfty} there exist $r', s \in \NN_0$ and coefficients
  $p'_\alpha \in \UnivStar(\lie g)_\subKnachKlammer^\inv$ for $\alpha \in \NN_0^{2d}$ with $\abs{\alpha} \le s$ 
  such that $\Pfaffian(K)^{r'} a = \sum_{\alpha \in \NN_0^{2d} , \abs{\alpha} \le s} p'_{\alpha} B_\subKnachKlammer^{\alpha}$.
  We will now show that $\Pfaffian(K)^s p'_\alpha \in I$ for all $\alpha \in \NN_0^{2d}$ with $\abs{\alpha} \le s$.
  This then completes the proof because it shows that \eqref{eq:commutatortrick} holds for $r \coloneqq r' + s$
  and $p_\alpha \coloneqq \Pfaffian(K)^s p'_\alpha$.
  
  So assume that for one $t \in \{0,\dots,s\}$ it has already been established that $\Pfaffian(K)^s p'_\alpha \in I$ for all 
  $\alpha \in \NN_0^{2d}$ with $s-t < \abs{\alpha} \le s$ (note that this is trivially true for $t=0$).
  Then for all $\alpha' \in \NN_0^{2d}$ with $\abs{\alpha'} = s-t$ Lemma~\ref{lemma:dual} shows that
  \begin{align*}
    \Pfaffian(K)^{s-\abs{\alpha'}} \ad_D^{\alpha'} \bigl(\Pfaffian(K)^{r'} a\bigr)
    &=
    \Pfaffian(K)^{s-\abs{\alpha'}} \sum_{\alpha \in \NN_0^{2d} , \abs{\alpha} \le s} p'_{\alpha} \ad_D^{\alpha'} \bigl(B_\subKnachKlammer^{\alpha}\bigr)
    \\
    &=
    \sum_{\substack{\alpha \in \NN_0^{2d} \\ \alpha' \le \alpha \textup{ and }\abs{\alpha}\le s}}
    \frac{\I^{\abs{\alpha'}} \alpha! }{(\alpha-\alpha')!}
    \Pfaffian(K)^s p'_\alpha B_\subKnachKlammer^{\alpha-\alpha'}
    .
  \end{align*}
  The left-hand side of this identity is an element of $I$ because $a\in I$. On the right-hand side, all summands except the one for $\alpha = \alpha'$
  are elements of $I$ because $\Pfaffian(K)^s p'_\alpha \in I$ for all $\alpha \in \NN_0^{2d}$ with $\abs{\alpha'} < \abs{\alpha} \le s$ by the induction assumption.
  It follows that the summand for $\alpha = \alpha'$ is an element of $I$, too, and after clearing the scalar prefactors this
  shows that $\Pfaffian(K)^s p'_{\alpha'} \in I$. By induction it follows that $\Pfaffian(K)^s p'_\alpha \in I$ for all $\alpha \in \NN_0^{2d}$ with $\abs{\alpha} \le s$.
\end{proof}

\begin{corollary} \label{corollary:commutatortrick}
	Let $K \in \FinSet$, let $I, J$ be ideals of $\UnivStar(\lie g)$ such that 
	$\Pfaffian(K) \bigl( \UnivStar(\lie g)_\subKnachKlammer^\inv \cap I \bigr) \subseteq J$,
	and assume that $J$ is real.
	Then $\Pfaffian(K) I \subseteq J$.
\end{corollary}

\begin{proof}
	Consider $a \in I$ and write $2d\in \NN_0$ for the number of elements in $K$. By the previous 
	Proposition~\ref{proposition:commutatortrick} there exist $r,s \in \NN_0$ such that $\Pfaffian(K)^r a = \sum_{\alpha \in \NN_0^{2d}, \abs \alpha \leq s} p_\alpha B_K^\alpha$
	with suitable coefficients $p_\alpha \in \UnivStar(\lie g)_\subKnachKlammer^\inv \cap I$.
	As $\Pfaffian(K)p_\alpha \in J$ for all $\alpha \in \NN_0^{2d}$ with $\abs{\alpha} \le s$ by assumption,
	it follows that $\Pfaffian(K)^{r+1} a \in J$.
	Now let $t \in \NN_0$ be such that $2^t \geq r+1$, then $\Pfaffian(K)^{2^t} a \in J$.
	We show recursively that $\Pfaffian(K)^{2^j} a \in J$ for all $j\in \{0,\dots,t-1\}$:
	Indeed, if $\Pfaffian(K)^{2^{j+1}} a \in J$ for one $j\in \{0,\dots,t-1\}$, then
	$(\Pfaffian(K)^{2^j} a)^* (\Pfaffian(K)^{2^j} a)= a^* \Pfaffian(K)^{2^{j+1}} a \in J$, and therefore $\Pfaffian(K)^{2^j} a \in J$
	because $J$ is real.
\end{proof}

\subsection{\texorpdfstring{$^*$\=/Representations}{*-Representations} from \texorpdfstring{$^*$\=/characters}{*-characters}}

In this section we are going to construct filtered $^*$\=/algebra morphisms from $\UnivStar(\lie g)$ to a Weyl algebra 
out of $^*$\=/characters on the $^*$\=/subalgebras $\UnivStar(\lie g)_\subKnachKlammer^\inv$ of $\UnivStar(\lie g)$.

\begin{proposition} \label{proposition:constructionOfRepFromFunctional}
  Assume $K = \{k_1,\dots,k_{2d}\} \in \FinSet$ with $d\in \NN_0$ and $k_1 < \dots < k_{2d}$,
  and $\varphi \in \charStar\bigl(\UnivStar(\lie g)_\subKnachKlammer^\inv\bigr)$.
  If $\varphi\bigl(\Pfaffian(K) \bigr) \neq 0$ and $\varphi\bigl(\Pfaffian(L) \bigr) = 0$ for all $L \in \FinSet$ with $K \subsetneq L$,
  then there is a filtered $^*$\=/algebra morphism $\Phi \colon \UnivStar(\lie g) \to \Weyl{d}$
  that fulfils $\Phi(p) = \varphi(p) \Unit$ for all $p \in \UnivStar(\lie g)_\subKnachKlammer^\inv$.
\end{proposition}
\begin{proof}
  Note that $\I \kom{X}{Y}$ for all $X,Y\in \lie g$ is a hermitian element of $\UnivStar(\lie g)$ so that $\varphi\bigl( \I \kom{X}{Y}\bigr) \in \RR$.
  Therefore the antisymmetric bilinear form $\omega \colon \genVS{\set{B_k}{k\in K}} \times \genVS{\set{B_k}{k\in K}} \to \RR$,
  \begin{equation*}
    (X,Y) \mapsto \varphi\bigl( \I \kom{X}{Y} \bigr)
  \end{equation*}
  is well-defined, where $\genVS{\argument}$ denotes the $\RR$-linear span in $\lie g$. In the basis $B_{k_1}, \dots, B_{k_{2d}}$ this bilinear
  form is described by the real matrix
  \begin{equation*}
    \Omega \coloneqq \bigl( \varphi(\I \kom{B_{k_j}}{B_{k_{j'}}})\bigr)_{j,j'=1}^{2d}
  \end{equation*}
  and $\det(\Omega) = \pfaffian(\Omega)^2 = \varphi\bigl(\Pfaffian(K)\bigr)^2 \neq 0$ because $\varphi$ is an algebra morphism.
  This shows that the antisymmetric form $\omega$ is non-degenerate and therefore we can construct a symplectic basis
  $X_1, \dots, X_d, Y_1 ,\dots, Y_d$ of $\genVS{\set{B_k}{k\in K}}$, i.e.
  \begin{equation*}
    \varphi\bigl( \I \kom{X_k}{X_\ell} \bigr) = \varphi\bigl( \I \kom{Y_k}{Y_\ell} \bigr) = 0
    \quad\quad\text{and}\quad\quad
    \varphi\bigl( \I \kom{X_k}{Y_\ell} \bigr) = -\Kronecker_{k,\ell}
  \end{equation*}
  for all $k,\ell \in \{1,\dots,d\}$ with $\Kronecker_{k,\ell}$ the Kronecker-$\delta$ (see e.g.~\cite[Sec.~1.3]{koszul.zou:IntroductionToSymplecticGeometry}).
  
  Recall that $\Pfaffian^k(K\cup \{m\})$ with $k\in K$ and $m\in \{1,\dots,\beta\} \setminus K$ is a hermitian element
  of $\UnivStar(\lie g)$ so that $\varphi\bigl(\Pfaffian^k(K\cup \{m\})\bigr) \in \RR$. We can therefore define
  \begin{equation*}
    F_m
    \coloneqq 
    \sign_K(m) \sum_{k\in K \cup \{m\}} \varphi\bigl(\Pfaffian^k(K\cup \{m\})\bigr)\, B_k \in \lie g
  \end{equation*}
  for all $m\in \{1,\dots,\beta\} \setminus K$. The coefficient in front of $B_m$ in this expansion is
  \begin{equation*}
    \sign_K(m) \varphi\bigl(\Pfaffian^m(K\cup \{m\})\bigr)
    =
    \sign_K(m) \sign_{K\cup \{m\}}(m) \varphi\bigl(\Pfaffian(K)\bigr)
    =
    \varphi\bigl(\Pfaffian(K)\bigr)
    \neq 0,
  \end{equation*}
  and therefore the basis vectors $X_1,\dots,X_d,Y_1,\dots,Y_d$ of $\genVS{\set{B_k}{k\in K}}$ together with
  these new vectors $F_m$ for $m\in \{1,\dots,\beta\} \setminus K$ form a basis of the linear subspace 
  $\genVS{\{B_1,\dots,B_\beta\}}$ of $\lie g$. Together with the basis $C_1,\dots,C_\gamma$ of the center $\lie z$ of $\lie g$
  we thus obtain a new basis of whole $\lie g$. Note the similarity between the definitions of $F_m \in \lie g$ (for fixed $K\in \FinSet$)
  and $\Invariant_{K,m} \in \UnivStar(\lie g)$.
  
  With the help of this new basis we define the $\RR$-linear map $\tilde \varphi \colon \lie g \to \Weyl{d}$,
  \begin{alignat*}{3}
    X_k     &\mapsto \tilde\varphi(X_k) &&\coloneqq \I P_k & &\text{for all $k\in \{1,\dots,d\}$},\\
    Y_\ell  &\mapsto \tilde\varphi(Y_\ell) &&\coloneqq \I Q_\ell & &\text{for all $\ell\in \{1,\dots,d\}$},\\
    F_m     &\mapsto \tilde\varphi(F_m) &&\coloneqq \varphi( \Invariant_{K,m})\Unit & &\text{for all $m\in \{1,\dots,\beta\} \setminus K$},\\
    C_j     &\mapsto \tilde\varphi(C_j) &&\coloneqq \varphi(C_j)\Unit&\quad\quad &\text{for all $j \in \{1,\dots,\gamma\}$}.
  \end{alignat*}
  We check that $\tilde\varphi$ fulfils the assumptions of Proposition~\ref{proposition:constructionOfRep}: First note that
  $\Invariant_{K,m}$ for $m\in \{1,\dots,\beta\} \setminus K$ and $C_j$ for $j \in \{1,\dots,\gamma\}$ are antihermitian,
  while all $P_k$, $Q_\ell$ with $k,\ell \in \{1,\dots,d\}$ and $\Unit$ are hermitian. Therefore 
  $\tilde\varphi$ maps all elements of $\lie g$ to antihermitian elements of $\Weyl{d}$. Moreover, 
  we have to check that $\tilde\varphi\bigl(\kom{V}{W}\bigr) = \kom[\big]{\tilde\varphi(V)}{\tilde\varphi(W)}$ holds for all $V,W \in \lie g$:
  First note that the Lie bracket of any two elements $V,W\in \lie g$ is central in $\lie g$ so that 
  $\tilde\varphi\bigl( \kom{V}{W}\bigr) = \varphi\bigl( \kom{V}{W}\bigr)\Unit$ by definition of $\tilde\varphi$.
  Clearly $\varphi\bigl(\kom{C_j}{W}\bigr)\Unit = 0 = \kom[\big]{\tilde\varphi(C_j)}{\tilde\varphi(W)}$ for $W\in \lie g$ and all $j\in \{1,\dots,\gamma\}$,
  and consequently also $\varphi\bigl(\kom{V}{C_j}\bigr)\Unit = \kom[\big]{\tilde\varphi(V)}{\tilde\varphi(C_j)}$ for $V\in \lie g$ and $j\in \{1,\dots,\gamma\}$.
  For the basis vectors $F_m$ with $m\in \{1,\dots,\beta\} \setminus K$ we find that
  \begin{align*}
    \varphi\bigl( \kom{F_m}{B_h} \bigr)
    &=
    \sign_K(m) \sum_{k\in K \cup \{m\}} \varphi\bigl(\Pfaffian^k(K\cup \{m\})\bigr)\, \varphi\bigl( \kom{B_k}{B_h} \bigr)
    \\
    &=
    -\I \sign_K(m) \,\varphi\bigl(\Pfaffian_h(K\cup \{m\})\bigr)
    \\
    &=
    0
  \end{align*}
  for all $h\in \{1,\dots,\beta\}$ by Proposition~\ref{proposition:PfaffianMatrixVectorOdd} and because $\varphi\bigl( \Pfaffian(L)\bigr) = 0$
  for all $L\in \FinSet$ with $K \subsetneq L$ by assumption.
  Therefore $\varphi\bigl(\kom{F_m}{W}\bigr)\Unit = 0 = \kom[\big]{\tilde\varphi(F_m)}{\tilde\varphi(W)}$ for $W\in \lie g$ and all $m\in \{1,\dots,\beta\} \setminus K$,
  and consequently also $\varphi\bigl(\kom{V}{F_m}\bigr)\Unit = \kom[\big]{\tilde\varphi(V)}{\tilde\varphi(F_m)}$ for $V\in \lie g$ and $m\in \{1,\dots,\beta\} \setminus K$.
  The remaining commutators of basis elements are easy to check:
  \begin{align*}
    \varphi\bigl(\kom{X_k}{X_\ell}\bigr)\Unit &= 0 = {\kom[\big]{\I P_k}{\I P_\ell}} = {\kom[\big]{\tilde\varphi(X_k)}{\tilde\varphi(X_\ell)}} \\
    \varphi\bigl(\kom{Y_k}{Y_\ell}\bigr)\Unit &= 0 = {\kom[\big]{\I Q_k}{\I Q_\ell}} = {\kom[\big]{\tilde\varphi(Y_k)}{\tilde\varphi(Y_\ell)}} \\
    \varphi\bigl(\kom{X_k}{Y_\ell}\bigr)\Unit &= \I \Kronecker_{k,\ell}\Unit =  {\kom[\big]{\I P_k}{\I Q_\ell}} = {\kom[\big]{\tilde\varphi(X_k)}{\tilde\varphi(Y_\ell)}}
  \end{align*}
  Therefore Proposition~\ref{proposition:constructionOfRep} applies and provides a $^*$\=/algebra morphism $\Phi \colon \UnivStar(\lie g) \to \Weyl{d}$
  that fulfils $\Phi \circ \iota = \tilde\varphi$ with $\iota \colon \lie g \to \UnivStar(\lie g)$ as in Definition~\ref{definition:iota}.  Clearly $\Phi$ is filtered.
  It is also clear that $\Phi(C_j) = \tilde\varphi(C_j) = \varphi(C_j) \Unit$ for all $j\in \{1,\dots,\gamma\}$.
  This already shows that $\Phi(p) = \varphi(p) \Unit$ for all $p\in \genSAlg{\lie z}$.
  Using this partial result we obtain for $m\in \{1,\dots,\beta\} \setminus K$:
  \begin{align*}
    \Phi(\Invariant_{K,m})
    &=
    \sign_K(m) \sum_{k\in K\cup\{m\}} \Phi \bigl( \Pfaffian^k(K\cup \{m\}) B_k\bigr)
    \\
    &=
    \sign_K(m) \sum_{k\in K\cup\{m\}} \varphi\bigl( \Pfaffian^k(K\cup \{m\}) \bigr) \Phi (B_k)
    \\
    &=
    \Phi(F_m)
    \\
    &=
    \tilde\varphi(F_m)
    \\
    &=
    \varphi(\Invariant_{K,m}) \Unit
    .
  \end{align*}
  Therefore $\Phi(p) = \varphi(p) \Unit$ holds for all $p\in \UnivStar(\lie g)_\subKnachKlammer^\inv$.
\end{proof}
This construction is illustrated in the example discussed in Section~\ref{sec:examples:nonparallel}.

\subsection{Proof of the main theorem} \label{sec:proofCompletion}

The following Lemma~\ref{lemma:theorem} shows that for the partially ordered set $(\FinSet, \subseteq)$ and
$\FinSetMapMap \coloneqq \Pfaffian \colon \FinSet \to \UnivStar(\lie g)$ the assumptions
of Proposition~\ref{proposition:stepwise} are fulfilled:

\begin{lemma} \label{lemma:theorem}
  Assume $K = \{k_1,\dots,k_{2d}\} \in \FinSet$ with $d\in \NN_0$ and $k_1 < \dots < k_{2d}$,
  and let $I$ be a real ideal of $\UnivStar(\lie g)$ such that $\Pfaffian(L) \in I$ for all $L\in \FinSet$ with $K \subsetneq L$.
  Write
  \begin{align}
  	Z_0 &\coloneqq \set[\big]{
  		\varphi \in \charStar\bigl( \UnivStar(\lie g)_\subKnachKlammer^\inv \bigr)
  	}{
  		\UnivStar(\lie g)_\subKnachKlammer^\inv \cap I \subseteq \ker \varphi 
  	},
  	\\
    S_0
    &\coloneqq
    \set[\big]{
      \varphi \in Z_0
    }{
      \varphi\bigl(\Pfaffian(K)\bigr) \neq 0    
    }
  \shortintertext{and}
    S
    &\coloneqq
    \set[\big]{
      \Phi \in \variety[d]
    }{
      \textup{there is $\varphi \in S_0$ such that $\Phi(p) = \varphi(p) \Unit$ for all $p \in \UnivStar(\lie g)_\subKnachKlammer^\inv$}
    }
    .
  \end{align}
  Then
  \begin{equation}
    \label{eq:lemma:main2}
    \Pfaffian(K) \,\Vanish{S} \subseteq I \subseteq \Vanish{S}
    .
  \end{equation}
\end{lemma}
\begin{proof}
  We begin with the following observation: The $^*$\=/algebra $\UnivStar(\lie g)_\subKnachKlammer^\inv$ is generated by
  the finite set $\set[\big]{\Invariant_{K,m}}{m\in \{1,\dots,\beta\} \setminus K} \cup \{C_1, \dots, C_\gamma\}$,
  and all commutators of these generators are elements of the real ideal $I$:
  this is immediately clear for commutators with the central elements $C_j$, $j\in \{1,\dots,\gamma\}$ (which are $0$),
  and identity \eqref{eq:commutators:invariant} of Proposition~\ref{proposition:commutators} shows that all commutators
  of $\Invariant_{K,m}$ with $m\in \{1,\dots,\beta\} \setminus K$ are elements of $I$ because
  $\Pfaffian\bigl( K \cup \{h,m\} \bigr) \in I$ for all $h \in \{1,\dots,\beta\}$ by assumption.
  Therefore the classical, commutative real Nullstellensatz in the form of Lemma~\ref{lemma:commutativeNullstellensatz} 
  applies to the real ideal $\UnivStar(\lie g)_\subKnachKlammer^\inv \cap I$ of $\UnivStar(\lie g)_\subKnachKlammer^\inv$
  and shows that
  \begin{equation*}
    \UnivStar(\lie g)_\subKnachKlammer^\inv \cap I
    =
    \Vanish[\big]{\ZerosNull{\UnivStar(\lie g)^\inv_\subKnachKlammer \cap I}}
   	=
    \Vanish{Z_0} .
  \end{equation*}
  Note that $\Vanish{X_0} \subseteq \UnivStar(\lie g)^\inv_\subKnachKlammer$ for $X_0 \subseteq \charStar(\UnivStar(\lie g)^\inv_\subKnachKlammer)$,
  and $\Vanish{X} \subseteq \UnivStar(\lie g)$ for $X \subseteq \variety[d]$.
  
  Next we show that $\Vanish{S} \cap \UnivStar(\lie g)^\inv_\subKnachKlammer = \Vanish{S_0}$:
  The inclusion $\supseteq$ follows immediately from the definition of $S$. Conversely,
  consider $a \in \Vanish{S} \cap \UnivStar(\lie g)^\inv_\subKnachKlammer$ and $\varphi \in S_0$.
  Then $\varphi(\Pfaffian(K)) \neq 0$ and $\UnivStar(\lie g)_\subKnachKlammer^\inv \cap I \subseteq \ker \varphi$
  so that $\varphi(\Pfaffian(L)) = 0$ for all $L \in \FinSet$ with $K \subsetneq L$.
  Therefore Proposition~\ref{proposition:constructionOfRepFromFunctional} applies and shows that there exists
  $\Phi \in \variety[d]$ such that $\Phi(p) = \varphi(p) \Unit$ for all $p\in \UnivStar(\lie g)_\subKnachKlammer^\inv$;
  so $\Phi \in S$ and then $\varphi(a) \Unit = \Phi(a) = 0$.

  Now note that $\varphi( \Pfaffian(K) a) = 0$ for all $a \in \Vanish{S_0}$ and $\varphi \in Z_0$, because by definition of $S_0$,
  either $\varphi(\Pfaffian(K)) = 0$ or $\varphi \in S_0$. From this and from the previous results it follows that
  \begin{gather*}
    \Pfaffian(K) \bigl(\Vanish{S} \cap \UnivStar(\lie g)^\inv_{\subKnachKlammer}\bigr)
    =
    \Pfaffian(K)\,\Vanish{S_0}
    \subseteq
    \Vanish{Z_0}
    =
    \UnivStar(\lie g)^\inv_{\subKnachKlammer} \cap I
    \subseteq
    I
  \shortintertext{and}
    \UnivStar(\lie g)_\subKnachKlammer^\inv \cap I
    =
    \Vanish{Z_0} \subseteq \Vanish{S_0}
    =
    \Vanish{S} \cap \UnivStar(\lie g)^\inv_\subKnachKlammer
    \subseteq
    \Vanish{S}
    .
  \end{gather*}
  By Corollary~\ref{corollary:commutatortrick} these two inclusions show that
  $\Pfaffian(K) \Vanish{S} \subseteq I$ and $\Pfaffian(K) I \subseteq \Vanish{S}$.
  
  Finally consider $a \in I$, $\Phi \in S$ and let $\varphi \in S_0$ be such that
  $\Phi(p) = \varphi(p) \Unit$ for all $p \in \UnivStar(\lie g)^\inv_\subKnachKlammer$.
  Then $\Phi(a) = \varphi(\Pfaffian(K))^{-1} \Phi( \Pfaffian(K)a) = 0$ because $\Pfaffian(K) I \subseteq \Vanish{S}$.
  This shows that $I \subseteq \Vanish{S}$.
\end{proof}

\begin{completionofproof}
  Consider the finite partially ordered set $(\FinSet, \subseteq)$ and the map
  $\FinSet \ni K \mapsto \FinSetMap{K} \coloneqq \Pfaffian(K) \in \UnivStar(\lie g)$.
  Then $\min \FinSet = \emptyset$ and $\FinSetMap{\min \FinSet} = \Pfaffian(\emptyset) = \Unit$. 
  An ideal $I$ of $\UnivStar(\lie g)$ is of type $K \in \FinSet$ in the sense of Proposition~\ref{proposition:stepwise}
  if and only if $\Pfaffian(L) \in I$ for all $L \in \FinSet$ with $K\subsetneq L$.
  For any tuple $(K,I)$ with $K \in \FinSet$ and $I$ a real ideal of $\UnivStar(\lie g)$ of type $K$, the previous 
  Lemma~\ref{lemma:theorem} provides a subset $S(K,I)$ of $\variety[\bullet]$ such that
  $\FinSetMap{K} \, \Vanish{S(K,I)} \subseteq I \subseteq \Vanish{S(K,I)}$ holds.
  Hence Proposition~\ref{proposition:stepwise} applies and shows that for every real ideal $I$ of $\UnivStar(\lie g)$
  there exists a subset $S$ of $\variety[\bullet]$ such that $\Vanish{S} = I$.
\end{completionofproof}

\section{Examples} \label{sec:examples}

\subsection{The Heisenberg Lie algebra} \label{sec:examples:heisenberg}

We continue the discussion of Example~\ref{example:heisenberg}, the Heisenberg Lie algebra $\lie h = \genVS{\{X,Y,Z\}}$.
Fix a real ideal $I \subseteq \UnivStar(\lie h)$.
In the proof of Theorem~\ref{theorem} we gain more and more information about $I$ in an iterative procedure,
which in the case of the Heisenberg Lie algebra consists of only two steps, corresponding to $K = \{1,2\}$ and $K = \emptyset$.
In each step we identify the correct filtered $^*$\=/algebra morphisms of the form $\Phi_\lambda \in \variety[1]$
and $\Psi_{\xi,\eta} \in \variety[0]$, respectively, that were discussed in Example~\ref{example:heisenberg}.
In the following we sketch the details of the general procedure in this specific case and in doing so derive a slightly stronger result.

In the first step, consider the commutative $^*$\=/subalgebra 
$\UnivStar(\lie h)^\inv_{\{1,2\}} = \genSAlg{\{Z\}} \subseteq \UnivStar(\lie h)$ and note that $\genSAlg{\{Z\}} \cong \CC[Z]$.
The intersection $I \cap \genSAlg{\{Z\}}$ clearly is a real ideal of $\genSAlg{\{Z\}}$.
In particular, it is generated by one minimal polynomial $\mu \in I \cap \genSAlg{\{Z\}}$.
By application of the fundamental theorem of algebra or the real Nullstellensatz, this polynomial
is of the form 
\begin{align}
  \label{eq:heisenberg:mu}
  \mu = {}&\prod\nolimits_{\lambda \in \Lambda} (\I Z - \lambda)
\intertext{with $\Lambda$ a finite (possibly empty) subset of $\RR$ (recall that $Z$ is antihermitian). Set $\Lambda^\times \coloneqq \Lambda \setminus \{0\}$ and}
  \label{eq:heisenberg:mured}
  \mu^\times \coloneqq {}&\prod\nolimits_{\lambda \in \Lambda^\times} (\I Z - \lambda)
  .
\end{align}
For all $\lambda \in \Lambda^\times$ we can construct the
filtered $^*$\=/algebra morphism $\Phi_\lambda \colon \UnivStar(\lie h) \to \Weyl{1}$ from Example~\ref{example:heisenberg}.
Its kernel is $\ker \Phi_\lambda = \genSId{\{\I Z - \lambda\}}$, the $^*$\=/ideal of $\UnivStar(\lie h)$ that is generated
by $\I Z - \lambda$. Clearly 
$I \cap \genSAlg{\{Z\}} \subseteq \bigcap\nolimits_{\lambda\in\Lambda^\times}\ker \Phi_\lambda$,
but it is not so obvious that $I \subseteq \bigcap\nolimits_{\lambda\in\Lambda^\times}\ker \Phi_\lambda$.
For this we need the following commutator trick:

\begin{lemma} \label{lemma:heisenbergCommutatorTrick}
  Every element $a \in I$ can be expanded as $a = \sum_{m,n=0}^\infty X^m Y^n q_{m,n} \mu^\times$
  with suitable coefficients $q_{m,n} \in \genSAlg{\{Z\}}$ (almost all of which are zero).
\end{lemma}
\begin{proof}
  This is mostly an application of Proposition~\ref{proposition:commutatortrick}. Essentially, writing $B_1 \coloneqq X$ and $B_2 \coloneqq Y$,
  then the ``dual'' elements $\Dual_1 \coloneqq Y$ and $\Dual_2 \coloneqq -X$ fulfil $\kom{B_k}{\Dual_\ell} = \Kronecker_{k,\ell} Z$
  for $k,\ell \in \{1,2\}$. Any element $a\in \UnivStar(\lie h)$ can be expanded as
  $a = \sum_{m,n=0}^{M,N} X^m Y^n p_{m,n}$ with suitable $p_{m,n} \in \genSAlg{\{Z\}}$ and $M,N \in \NN_0$. If $a\in I$, then also all commutators
  with $a$ are in $I$ and by applying commutators with these ``dual'' elements one iteratively finds that $Z^{M+N} p_{m,n} \in I \cap \genSAlg{\{Z\}}$
  for all $m,n \in \NN_0$, so $Z^{M+N}p_{m,n}$ is a multiple of $\mu$. As $0$ is not a root of $\mu^\times$,
  it follows that $p_{m,n}$ is a multiple of $\mu^\times$,
  i.e.\ there are coefficients $q_{m,n} \in \genSAlg{\{Z\}}$ such that $p_{m,n} = q_{m,n} \mu^\times$.
\end{proof}
For every element $a \in I$ there thus exists a unique element $\frac{a}{\mu^\times} \in \UnivStar(\lie h)$ such that $\frac{a}{\mu^\times} \mu^\times = a$,
and especially $I\subseteq \genSId{\{\mu^\times\}} \subseteq \bigcap_{\lambda \in \Lambda^\times} \ker \Phi_\lambda$,
where $\genSId{\argument}$ denotes the generated $^*$\=/ideal in $\UnivStar(\lie h)$.
Conversely also $\genSId{\{\mu^\times\}} \supseteq \bigcap_{\lambda \in \Lambda^\times} \ker \Phi_\lambda$ holds,
in order to see this we repeat the above discussion with $J \coloneqq \bigcap_{\lambda \in \Lambda^\times} \ker \Phi_\lambda$ in place of $I$:
The real ideal $J \cap \genSAlg{\{Z\}}$ of $\genSAlg{\{Z\}}$ is generated by the minimal polynomial $\mu^\times$.
Application of Lemma~\ref{lemma:heisenbergCommutatorTrick} now shows that all elements of $J$ are multiples of $\mu^\times$,
so $J \subseteq \genSId{ \{\mu^\times\}}$.

From \eqref{eq:heisenberg:mu} we also have $Z \mu^\times \in \genSId{\{\mu\}} \subseteq I$, hence
\begin{equation}
  \label{eq:heisenberg:vorQuotient}
  Z \,\genSId[\big]{\{\mu^\times\}} \subseteq I \subseteq \genSId[\big]{\{\mu^\times\}}
\end{equation}
with $\genSId{\{\mu^\times\}} = \bigcap_{\lambda \in \Lambda^\times} \ker \Phi_\lambda$.
This gives a good first approximation of $I$.
Note that \eqref{eq:heisenberg:vorQuotient} is essentially \eqref{eq:stepwise} from Proposition~\ref{proposition:stepwise}.
In the second step, we wish to improve \eqref{eq:heisenberg:vorQuotient}. To this end, consider
\begin{equation}
  I' \coloneqq \set[\bigg]{\frac{a}{\mu^\times}}{a\in I} = \set[\big]{b \in \UnivStar(\lie h)}{b \mu^\times \in I}.
\end{equation}
In the language of Section~\ref{sec:realIdealQuotients}, $I'$ is the ideal quotient of $I$ over $\genSId{\{\mu^\times\}}= \bigcap_{\lambda \in \Lambda^\times} \ker \Phi_\lambda$.
It is not hard to check that $I'$ is again a real ideal of $\UnivStar(\lie h)$ (see also Lemma~\ref{lemma:realIdealQuotients}),
and $I' \mu^\times = I$ by construction.

Clearly $Z \in I'$ because $\mu \in I$, so $\genSId{\{Z\}} \subseteq I'$ and the quotient $^*$\=/algebra $\UnivStar(\lie h) / \genSId{\{Z\}}$
is commutative. In this quotient, $X$ and $Y$ are thus central, i.e.~they are ``invariant'' in the language of Section~\ref{sec:invariantAndDual}.
We could directly apply the real Nullstellensatz in the form of Lemma~\ref{lemma:commutativeNullstellensatz} to $I'$, but a more
explicit construction might provide more insight: We consider the canonical isomorphism $\UnivStar(\lie h) / \genSId{\{Z\}} \cong \CC[X,Y]$
and write $[\argument] \colon \UnivStar(\lie h)  \to \CC[X,Y]$
for the canonical projection onto this quotient. 
We construct a \emph{linear} split $s \colon \CC[X,Y] \to \UnivStar(\lie h)$,
\begin{equation}
  X^m Y^n \mapsto s(X^mY^n) \coloneqq X^m Y^n \quad\quad\text{for all $m,n\in \NN_0$,}
\end{equation}
i.e.\ $[s(p)] = p$ for all $p\in \CC[X,Y]$. Of course, $s$ is not a $^*$\=/algebra morphism.
Define the real ideal $I^\downarrow \coloneqq \set{[a]}{a\in I'} \cong I' / \genSId{\{Z\}}$ in $\CC[X,Y] \cong \UnivStar(\lie h) / \genSId{\{Z\}}$,
then the real Nullstellensatz shows that $I^\downarrow = \Vanish[\big]{\mathcal{Z}_0(I^\downarrow)}$
for a subset $\mathcal{Z}_0(I^\downarrow)$ of $\RR^2$.
It follows that $I' = s\bigl(\Vanish[\big]{\mathcal{Z}_0(I^\downarrow)}\bigr) + \genSId{\{Z\}}$.
For all $(\xi,\eta) \in \mathcal{Z}_0(I^\downarrow)$ we can construct the filtered $^*$\=/algebra morphisms $\Psi_{\xi,\eta} \colon \UnivStar(\lie h) \to \Weyl{0}$,
which map $Z$ to zero and then evaluate at $(\xi,\eta) \in \RR^2$, and so we get
\begin{equation}
  \label{eq:Iprime}
  I' = \bigcap\nolimits_{(\xi,\eta) \in \mathcal{Z}_0(I^\downarrow)} \ker \Psi_{\xi,\eta}
  .
\end{equation}
In total we thus obtain:
\begin{proposition}
  An ideal $I$ of $\UnivStar(\lie h)$ is real if and only if there exist a finite subset $\Lambda^\times$ of $\RR\setminus\{0\}$
  and a real algebraic subset $\mathcal{N}$ of $\RR^2$ (i.e.~$\mathcal{N}$ is the set of zeros of a polynomial) such that
  \begin{align}
    \label{eq:theorem:heisenberg2}
    I &= \Bigl(\bigcap\nolimits_{(\xi,\eta) \in \mathcal{N}} \ker \Psi_{\xi,\eta}\Bigr)\cap\Bigl(\bigcap\nolimits_{\lambda \in \Lambda^\times} \ker \Phi_\lambda\Bigr)
    .
  \shortintertext{In this case} 
    \label{eq:theorem:heisenberg1}
    I &= \bigl( s(\Vanish{\mathcal{N}}) + \genSId{\{Z\}} \bigr) \Bigl(\prod\nolimits_{\lambda \in \Lambda^\times} (\I Z - \lambda)\Bigr)
    .
  \end{align}
\end{proposition}
\begin{proof}
  For any real ideal $I$ of $\UnivStar(\lie h)$ the above discussion shows that
  $I = I' \mu^\times$ with $I'$ as in \eqref{eq:Iprime},
  so $I$ is of the form \eqref{eq:theorem:heisenberg1} with $\mathcal{N} \coloneqq \mathcal{Z}_0(I^\downarrow)$.
  If \eqref{eq:theorem:heisenberg2} holds, then $I$ is real by Proposition~\ref{proposition:vanishingIsReal}.
  It remains to show that \eqref{eq:theorem:heisenberg2} holds for every real ideal $I$, which follows from 
  Lemma~\ref{lemma:realIdealQuotients} (like in the proof of Proposition~\ref{proposition:stepwise}):
  \begin{equation*}
    I
    =
    I \cap \genSId[\big]{\{\mu^\times\}}
    =
    \bigl( I : \genSId[\big]{\{\mu^\times\}} \bigr) \cap \genSId[\big]{\{\mu^\times\}}
    =
    I' \cap \genSId[\big]{\{\mu^\times\}}
  \end{equation*}
  where the second identity is due to Lemma~\ref{lemma:realIdealQuotients};
  we have already seen that $I' = \bigcap\nolimits_{(\xi,\eta) \in \mathcal{Z}_0(I^\downarrow)} \ker \Psi_{\xi,\eta}$
  and $\genSId{\{\mu^\times\}} = \bigcap_{\lambda\in\Lambda^\times} \ker \Phi_\lambda$.
\end{proof}

\subsection[\texorpdfstring{Free $2$-step nilpotent Lie algebras with $3$ generators}{Free 2-step nilpotent Lie algebras with 3 generators}]{\boldmath Free $2$-step nilpotent Lie algebras with $3$ generators} \label{sec:examples:nonparallel}

As a last example we consider the $2$-step nilpotent Lie algebra $\fthree$ that is generated freely by three elements $B_1,B_2,B_3$,
i.e.\ $\fthree$ is the $6$-dimensional real Lie algebra with basis $B_1,B_2,B_3,C_1,C_2,C_3$, where $C_1,C_2,C_3$ are central and
\begin{equation}
  {\kom{B_1}{B_2}} \coloneqq C_3, \quad\quad {\kom{B_2}{B_3}} \coloneqq C_1, \quad\quad\text{and}\quad\quad {\kom{B_3}{B_1}} \coloneqq C_2.
\end{equation}
In contrast to the Heisenberg Lie algebra $\lie h$, the coadjoint orbits of $\fthree$ are not parallel, which means that
its representation space has a more complicated geometry. This example thus allows us to discuss some of the additional 
complications occurring in the general proof of Theorem~\ref{theorem}. Let $I$ be a real $^*$\=/ideal of $\UnivStar(\fthree)$.
Note that 
\begin{equation}
  \FinSet = \bigl\{ \{2,3\}, \{1,3\}, \{1,2\}, \emptyset \bigr\}
  .
\end{equation}
The iterative construction of Proposition~\ref{proposition:stepwise} divides the problem of finding
a set $S \subseteq \variety[\bullet]$ such that $I = \Vanish{S}$ into four smaller subproblems dealing with
real ideals of type $K \in \FinSet$. 

First consider $K \in \bigl\{ \{2,3\}, \{1,3\}, \{1,2\} \bigr\}$, then every ideal $I$ is of type $K$.
In these three steps we want to identify the (non-trivial) filtered $^*$\=/algebra morphisms
$\Phi \colon \UnivStar(\fthree) \to \Weyl{d}$ with $d=1$ that fulfil $I \subseteq \ker \Phi$.
Naively, one could try to obtain $\Phi$ as an extension of a $^*$\=/character $\varphi$
whose domain is the commutative $^*$\=/subalgebra $\Phi^{-1}\bigl( \genSAlg{\{\Unit\}}\bigr)$
of $\UnivStar(\fthree)$.
However, in contrast to the first example of the Heisenberg Lie algebra, this preimage
$\Phi^{-1}\bigl( \genSAlg{\{\Unit\}}\bigr)$ now depends on $\Phi$! We therefore
have to restrict to a sufficiently large common $^*$\=/subalgebra $\UnivStar(\fthree)^\inv$ of all these different
commutative $^*$\=/subalgebras, then we can identify the correct $^*$\=/characters by applying the commutative real
Nullstellensatz Lemma~\ref{lemma:commutativeNullstellensatz} to $I \cap \UnivStar(\fthree)^\inv$.

The naive choice $\UnivStar(\fthree)^\inv \!\coloneqq \genSAlg{ \{C_1,C_2,C_3\}}$
is too small, a $^*$\=/character on $\genSAlg{ \{C_1,C_2,C_3\}}$ does not carry enough information to reconstruct 
a filtered $^*$\=/algebra morphism $\Phi \colon \UnivStar(\fthree) \to \Weyl{1}$.
Loosely speaking, the maximal coadjoint orbits have codimension $4$, 
so that we need $4$ generators.
Luckily there exists another independent central element in $\UnivStar(\fthree)$, which however is not in $\fthree$, namely
\begin{equation}
  \Invariant \coloneqq B_1 C_1 + B_2 C_2 + B_3 C_3
  .
\end{equation}
It is easy to check that indeed $\kom{\Invariant}{B_j} = 0$ for all $j\in\{1,2,3\}$, and comparison with Definition~\ref{definition:generator}
shows that $\I \Invariant = \Invariant_{\{2,3\},1} = - \Invariant_{\{1,3\},2} = \Invariant_{\{1,2\},3}$. We therefore set
\begin{equation}
  \UnivStar(\fthree)^\inv \coloneqq \genSAlg[\big]{\{\Invariant, C_1,C_2,C_3\}}
\end{equation}
and can then apply the commutative real Nullstellensatz to the real $^*$\=/ideal $I \cap \UnivStar(\fthree)^\inv$ of $\UnivStar(\fthree)^\inv$.
This $^*$\=/subalgebra $\UnivStar(\fthree)^\inv$ coincides with each of the three $^*$\=/subalgebra $\UnivStar(\fthree)^\inv_K$
of $\UnivStar(\fthree)$ from Definition~\ref{definition:UnivInv} for $K \in \bigl\{ \{2,3\}, \{1,3\}, \{1,2\} \bigr\}$.

Any $^*$\=/character $\varphi \in \charStar\bigl( \UnivStar(\fthree)^\inv \bigr)$ that fulfils $\varphi(C_j) \neq 0$
for at least one $j\in \{1,2,3\}$ can be extended to a filtered $^*$\=/algebra morphism $\Phi \colon \UnivStar(\fthree) \to \Weyl{1}$
as in Proposition~\ref{proposition:constructionOfRepFromFunctional}:
E.g.~if $K = \{2,3\}$ and $\varphi\bigl( \Pfaffian(K) \bigr) = \varphi(\I C_1) \neq 0$, then we can choose a symplectic basis
of $\genVS{\{B_2,B_3\}}$ as $X \coloneqq B_2$ and $Y \coloneqq - \varphi(\I C_1) B_3$.
Note that $C_j$ is antihermitian so that $\varphi( \I C_j ) \in \RR$ for all $j\in \{1,2,3\}$.
Together with $F \coloneqq \varphi( \I C_1 ) B_1 + \varphi(\I C_2) B_2 + \varphi(\I C_3) B_3 \in \lie g$
we obtain a basis $X,Y,F,C_1,C_2,C_3$ of $\lie g$. Then
\begin{align}
  \Phi(X) &\coloneqq \I P & \Phi(Y) &\coloneqq \I Q & \Phi(F) &\coloneqq \varphi(E) \Unit \\
  \Phi(C_1) &\coloneqq \varphi(C_1) \Unit & \Phi(C_2) &\coloneqq \varphi(C_2) \Unit & \Phi(C_3) &\coloneqq \varphi(C_3) \Unit 
\end{align}
extends in a unique way to a filtered $^*$\=/algebra morphism $\Phi \colon \UnivStar(\fthree) \to \Weyl{1}$.
Note that this extension $\Phi$ is not uniquely determined by $\varphi$.
For $K \in \bigl\{\{1,3\},\{1,2\}\bigr\}$ the construction is analogous.

The last step for $K = \emptyset$ then deals only with real ideals $I$ of type $\emptyset$,
meaning that $\{C_1,C_2,C_3\} \subseteq I$ so that $\UnivStar(\lie f_3) / I$ is commutative.
In this case the commutative real Nullstellenatz Lemma~\ref{lemma:commutativeNullstellensatz}
applies directly and provides the required filtered $^*$\=/algebra morphisms $\Phi \colon \UnivStar(\lie f_3) \to \Weyl{0} \cong \CC$.

For larger $2$-step nilpotent real Lie algebras $\lie g$ it will also happen that there are additional steps in between the first ones where
$\UnivStar(\lie g)^\inv$ is generated by central elements of $\UnivStar(\lie g)$, and the last one where the ideal $I$ is of type $\emptyset$
and therefore contains the whole center of $\lie g$ so that $\UnivStar(\lie g) / I$ is commutative.

\section*{Acknowledgements}

Most progress on this article and its follow-up has been made during mutual visits for
various occasions in different places, and we would like to thank the respective organizers and hosts:
the 2022 workshop ``math in the mill'' in Sondheim vor der Rhön; IMPAN Warsaw; and Leibniz University Hannover. 
This research is also part of the EU Staff Exchange project 101086394 ``Operator Algebras That One Can See''.
It was partially supported by the University of Warsaw Thematic Research Programme ``Quantum Symmetries''.
The second author would like to thank Profs.~K.~Schmüdgen and J.~Cimpri\v{c}
for insightful discussions and their encouragement to work on this topic.

\end{onehalfspace}

{
	\normalfont\footnotesize

}

\end{document}